\documentclass[12pt]{article}
\usepackage{amsmath,amsthm,amsfonts,amssymb}
\usepackage[pdftex,pdfborder={0 0 0}]{hyperref}
\usepackage{fullpage}
\usepackage{multirow}
\usepackage{array}
\usepackage{bbm}
\usepackage[noblocks]{authblk}
\usepackage{verbatim}
\usepackage{tikz}
\usepackage{float}
\usepackage{enumerate}
\usepackage{diagbox}
\usepackage{soul}
\usepackage{cases}

\newtheorem{theorem}{Theorem}[section]
\newtheorem{lemma}[theorem]{Lemma}

\newtheorem{corollary}[theorem]{Corollary}
\newtheorem{conjecture}[theorem]{Conjecture}

\theoremstyle{definition}
\newtheorem{definition}[theorem]{Definition}
\newtheorem{example}[theorem]{Example}

\newlength{\Oldarrayrulewidth}

\newcommand{\N}{\mathbb{N}}

\newcommand{\Z}{\mathbb{Z}}

\newcommand{\e}{\textup{\textbf{e}}}

\renewcommand{\i}{\textup{\textbf{i}}}
\renewcommand{\j}{\textup{\textbf{j}}}

\newcommand{\n}{\textup{\textbf{n}}}

\renewcommand{\r}{\textup{\textbf{r}}}

\newcommand{\x}{\textup{\textbf{x}}}

\newcommand{\bz}{\textup{\textbf{0}}}

\renewcommand{\mod}[2]{\equiv#1\textup{ (mod }#2\textup{)}}

\def\m@th{\mathsurround=0pt}
\def\sm#1{\null\,\vcenter{\baselineskip9pt\lineskip.23ex\m@th
    \ialign{\hfil$\scriptstyle##$\hfil&&\ \hfil$\scriptstyle##$\hfil\crcr
    \mathstrut\crcr\noalign{\kern-\baselineskip}
    #1\crcr\mathstrut\crcr\noalign{\kern-\baselineskip}}}\,}
\def\smnp#1{\null\,\vcenter{\baselineskip9pt\lineskip.23ex\m@th
    \ialign{\hfil$\scriptstyle##$\hfil&&\ \ \hfil$\scriptstyle##$\hfil\crcr
    \mathstrut\crcr\noalign{\kern-\baselineskip}
    #1\crcr\mathstrut\crcr\noalign{\kern-\baselineskip}}}\,}

\begin{document}

\title{Sum Index and Difference Index of Graphs}

\author[1]{Joshua Harrington\thanks{joshua.harrington@cedarcrest.edu}}
\affil[1]{Department of Mathematics, Cedar Crest College}

\author[2]{Eugene Henninger-Voss\thanks{eugene.henninger\_voss@tufts.edu}}
\affil[2]{Department of Mathematics, Tufts University}

\author[3]{Kedar Karhadkar\thanks{kjk5800@psu.edu}}
\affil[3]{Department of Mathematics, The Pennsylvania State University}

\author[4]{Emily Robinson\thanks{robin28e@mtholyoke.edu}}
\affil[4]{Department of Mathematics and Statistics, Mount Holyoke College}

\author[5]{Tony W.\ H.\ Wong\thanks{wong@kutztown.edu}}
\affil[5]{Department of Mathematics, Kutztown University of Pennsylvania}

\date{\today}

\maketitle

\begin{abstract}
Let $G$ be a nonempty simple graph with a vertex set $V(G)$ and an edge set $E(G)$. For every injective vertex labeling $f:V(G)\to\mathbb{Z}$, there are two induced edge labelings, namely $f^+:E(G)\to\Z$ defined by $f^+(uv)=f(u)+f(v)$, and $f^-:E(G)\to\Z$ defined by $f^-(uv)=|f(u)-f(v)|$. The sum index and the difference index are the minimum cardinalities of the ranges of $f^+$ and $f^-$, respectively. We provide upper and lower bounds on the sum index and difference index, and determine the sum index and difference index of various families of graphs. We also provide an interesting conjecture relating the sum index and the difference index of graphs.\\
\textit{Keywords}: Graph labeling; Sum index; Difference index.\\
\textit{MSC}: 05C78, 05C05.
\end{abstract}

\section{Introduction}\label{sec:intro}

Throughout this paper, let $G$ denote a nonempty simple graph with vertex set $V(G)$ and edge set $E(G)$.  A \emph{vertex labeling} of $G$ is an injective map $f:V(G)\to\Z$. Let $f^+:E(G)\to\Z$ be the induced \emph{edge labeling} defined by $f^+(uv)=f(u)+f(v)$ for each edge $uv\in E(G)$. Note that an edge labeling function is not necessarily injective. For every edge labeling $g:E(G)\to\Z$, define $|g|=|g(E(G))|$ as the cardinality of the range of $g$, which counts the number of distinct edge labels assigned by $g$.

The notion of inducing edge labelings by summing the labels of the incident vertices has been studied in different contexts. For example, this notion was used by Harary \cite{harary} to introduce sum labelings and sum graphs. It was also used by Ponraj and Parthipan \cite{pp} to introduce pair sum labelings and pair sum graphs. More recently, Harrington and Wong \cite{hw} used this notion to introduce the following definition of the sum index of $G$.

\begin{definition}\label{def:sumindex}
The \emph{sum index} of $G$, denoted by $s(G)$, is the minimum positive integer $k$ such that there exists a vertex labeling $f$ of $G$ with $|f^+|=k$. A vertex labeling $f$ such that $|f^+| = s(G)$ is referred to as a \emph{sum index labeling} of $G$.
\end{definition}

Harrington and Wong proved that $s(G)\geq\Delta(G)$, where $\Delta(G)$ denotes the maximum degree of $G$. They also showed that if $n\geq2$, then $s(K_n)=2n-3$, thus $\Delta(G)\leq s(G)\leq 2n-3$ for any graph $G$ with $n$ vertices.  Furthermore, they determined that $s(K_{n,m})=n+m-1$ for complete bipartite graphs $K_{n,m}$ and $s(G)=\Delta(G)$ if $G$ is a caterpillar graph. Lastly, they studied the sum index of trees and showed that if the diameter of a tree $T$ is at most $5$, then $s(T)=\Delta(T)$, but they proved that this equality does not hold for all trees in general.

In this article, we slightly improve the lower bound of the sum index by using the chromatic index $\chi'(G)$ and provide several upper bounds in Subsection~\ref{subsec:sumbounds}. In Subsection~\ref{subsec:sumfamilies}, we determine the sum index of graphs in the following families: cycles, spiders, wheels, and $d$-dimensional rectangular grids.  Further, we show that $s(G)-\chi'(G)$ can be arbitrarily large, as exhibited by several families of graphs, such as trees and triangular grids.  We construct graphs with a prescribed sum index in Subsection~\ref{subsec:construct}, and end our Section~\ref{sec:sumindex} by analyzing some underlying structures of trees with a fixed upper bound on its sum index.

Closely resembling the definition of sum index, we define the difference index as follows.

\begin{definition}\label{def:diffindex}
Let $f:V(G)\to\Z$ be a vertex labeling of $G$, and let $f^-:E(G)\to\Z$ be the induced edge labeling defined by $f^-(uv)=|f(u)-f(v)|$ for each edge $uv\in E(G)$. The \emph{difference index} of $G$, denoted by $d(G)$, is the minimum positive integer $k$ such that there exists a vertex labeling $f$ of $G$ with $|f^-|=k$. A vertex labeling $f$ such that $|f^-|=d(G)$ is referred to as a \emph{difference index labeling} of $G$.
\end{definition}

As an analogue to Section~\ref{sec:sumindex}, we provide several bounds on the difference index, determine the difference index of various families of graphs, and analyze trees with a fixed upper bound on its difference index in Section~\ref{sec:diffindex}. We conclude our paper with two conjectures. In particular, we conjecture that the difference index is half of the sum index for all nonempty simple graphs.

\section{Sum index}\label{sec:sumindex}

We begin our study of sum index by providing upper and lower bounds for $s(G).$

\subsection{Bounds on the sum index}\label{subsec:sumbounds}

As mentioned in the introduction, the maximum degree $\Delta(G)$ is a lower bound of $s(G)$. The following theorem slightly improves the lower bound of $s(G)$ by using $\chi'(G)$, the chromatic index of $G$.

\begin{theorem}\label{thm:chromindex}
The sum index is greater than or equal to the chromatic index, i.e., $s(G)\geq\chi'(G)$.
\end{theorem}

\begin{proof}
Let $f$ be a sum index labeling of $G$. Since $f$ is injective, we may view $f^+$ as a proper edge coloring of $G$. Indeed, if two incident edges $uv$ and $uw$ share the same edge label, i.e., $f^+(uv) = f^+(uw)$, then $f(u)+f(v)=f(u)+f(w)$, so $f(v)=f(w)$. This contradicts the injectivity of $f$. Thus, $f$ induces a proper edge coloring on $G$ with $|f^+|$ colors, so $s(G)=|f^+|\geq\chi'(G)$.
\end{proof}

Before we provide upper bounds for $s(G)$, we first introduce the definitions of sum labeling and exclusive sum labeling. Here, $\overline{K_k}$ denotes the complement of the complete graph $K_k$, which is the graph of $k$ isolated vertices.

\begin{definition}
A \textit{sum labeling} of a graph $G$ is an injective map $f: V(G) \rightarrow \mathbb{N}$ such that two vertices $v, w \in V(G)$ are adjacent if and only if $f(v) + f(w) = f(u)$ for some vertex $u \in V(G)$. If $G$ admits a sum labeling, then $G$ is a \textit{sum graph}. The \textit{sum number} $\sigma(G)$ is the minimum nonnegative integer $k$ such that $G\cup\overline{K_k}$ is a sum graph.
\end{definition}

\begin{definition}
Let $k$ be a positive integer. A \textit{$k$-exclusive sum labeling} (abbreviated $k$-ESL) of a graph $G$ is an injective map $f: V(G \cup \overline{K_k}) \to \mathbb{N}$ such that two vertices $v,w\in V(G)$ are adjacent if and only if $f(v)+f(w)=f(u)$ for some vertex $u\in V(\overline{K_k})$. The \textit{exclusive sum number} $\epsilon(G)$ is the minimum $k$ such that $G$ admits a $k$-ESL.
\end{definition}

The sum number of $G$ was introduced by Harary \cite{harary} and the exclusive sum number of $G$ was introduced by Miller et.\ al.\ \cite{mprsst}. With these two definitions in mind, the following theorems give two upper bounds of the sum index $s(G)$.

\begin{theorem}
Let $u$ be a vertex of $G$, and let $G_u$ be the induced subgraph $G \setminus\{u\}$. If $G$ has $n$ vertices, then $s(G)\leq\min_{u\in V(G)}\{n-1+\sigma(G_u)\}$.
\end{theorem}

\begin{proof}
For every vertex $u\in V(G)$, let $H_u=G_u\cup\overline{K_{\sigma(G_u)}}$, which is a sum graph by definition. Let $\widetilde{f}$ be a sum labeling of $H_u$. Define an injective vertex labeling $f:V(G)\to\Z$ such that $f(u)=0$ and $f(v)=\widetilde{f}(v)$ for all $v\in V(G_u)$. For all $vw\in E(G_u)$, $f^+(vw) =\widetilde{f}(x)$ for some $x\in V(H_u)$, and for all $uv\in E(G)$, $f^+(uv)=f(u)+f(v)=\widetilde{f}(v)$. This implies that $f^+(E(G))\subseteq\widetilde{f}(V(H_u))$. Thus,
$$s(G)\leq|f^+|\leq|V(H_u)|=|V(G_u)|+\sigma(G_u)=n-1+\sigma(G_u).$$
\end{proof}

\begin{theorem}\label{thm:esnbound}
The sum index is less than or equal to the exclusive sum number, i.e., $s(G)\leq\epsilon(G)$. Moreover, there exists a graph $G$ such that $s(G)<\epsilon(G)$.
\end{theorem}

\begin{proof}
Let $k=\epsilon(G)$, and let $g$ be a $k$-ESL of $G$. Then $g$ restricts to a vertex labeling $f$ of $G$ such that $|f^+|=|V(\overline{K_k})|=k$. As a result,
$$s(G)\leq|f^+|=k=\epsilon(G).$$

\begin{figure}[H]
    \centering
    \begin{tikzpicture}[scale=1]
    \coordinate(v0)at(0,0);\coordinate(v1)at(1,0); \coordinate(v2)at({1+cos(60)},{sin(60)});\coordinate(v3)at({2+cos(60)},{sin(60)});\coordinate(v4)at(1,{2*sin(60)});\coordinate(v5)at({cos(120)},{sin(120)});\coordinate(v6)at({-1+cos(120)},{sin(120)});\coordinate(v7)at(0,{2*sin(120)});\coordinate(v8)at({cos(-120)},{sin(-120)});\coordinate(v9)at({-1+cos(-120)},{sin(-120)});\coordinate(v10)at(0,{2*sin(-120)});\foreach\i in{v0,v1,v2,v3,v4,v5,v6,v7,v8,v9,v10}{\filldraw(\i)circle(0.1);};\draw(v6)--(v5)--(v0)--(v1)--(v2)--(v3);\draw(v7)--(v5);\draw(v9)--(v8)--(v0);\draw(v8)--(v10);\draw(v4)--(v2);\node[below right]at(v0){$v_0$};\node[below right]at(v1){$v_1$};\node[left]at(v2){$v_2$};\node[above right]at(v3){$v_3$};\node[above right]at(v4){$v_4$};\node[above left]at(v5){$v_5$};\node[above left]at(v6){$v_6$};\node[above right]at(v7){$v_7$};\node[right]at(v8){$v_8$};\node[above left]at(v9){$v_9$};\node[below right]at(v10){$v_{10}$};
    \end{tikzpicture}
    \caption{A graph $G$ with $s(G)=3$ but $\epsilon(G)>3$.}
    \label{fig:exclsumnum}
\end{figure}
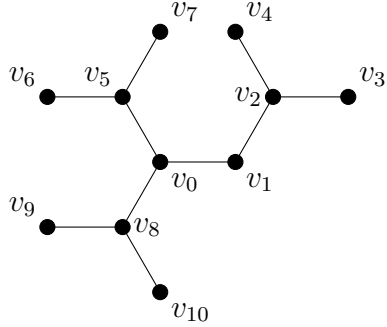

Now, we prove that the graph $G$ in Figure~$\ref{fig:exclsumnum}$ satisfies $s(G)<\epsilon(G)$. By Theorem $4.9$ in the paper by Harrington and Wong, since $G$ is a tree with maximum degree $3$ and diameter $5$, the sum index $s(G)=\Delta(G)=3$.  

For the sake of contradiction, assume that there is a $3$-ESL $g$ of $G$. Let the image of $g^+$ be the set $\{\alpha,\beta,\gamma\}$. As proved in Theorem \ref{thm:chromindex}, since $g$ is injective, $g^+$ is a proper edge coloring on $G$. Hence, every degree $3$ vertex in $G$ must be incident to edges with all three labels $\alpha$, $\beta$, and $\gamma$. 
Without loss of generality, assume that $g^+(v_0v_1) = \alpha$ and $g^+(v_1v_2)=\beta$. Since $v_2$ is incident to edges with all three labels, let $g^+(v_2v_3)=\gamma$. Since the subgraph induced by $\{v_0,v_5,v_6,v_7\}$ is isomorphic to the subgraph induced by $\{v_0,v_8,v_9,v_{10}\}$, we may assume that $g^+(v_0v_5)=\gamma$. Since $v_5$ is incident to edges with all three labels, let $g^+(v_5v_6)=\beta$. All assumptions are made without loss of generality.

When we consider the labels of the vertices, note that $g(v_0)+g(v_1)=\alpha$, $g(v_2)+g(v_3)=\gamma$, and $g(v_5)+g(v_6)=\beta$, so
$$g(v_0)+g(v_1)+g(v_2)+g(v_3)+g(v_5)+g(v_6)=\alpha+\beta+\gamma.$$
Also, since $g(v_0)+g(v_5)=\gamma$ and $g(v_1)+g(v_2)=\beta$, we have
$$g(v_0)+g(v_1)+g(v_2)+g(v_5)=\beta+\gamma.$$
As a result, $g(v_3)+g(v_6)=\alpha$ by subtraction. However, the edge $v_3v_6$ is not in $G$, contradicting the definition of a $3$-ESL. Therefore, there is no $3$-ESL of $G$, i.e., $\epsilon(G)>3$.
\end{proof}

After establishing both lower and upper bounds for the sum index, it is natural to ask how good these bounds are. In other words, we would like to study the magnitudes of $s(G)-\chi'(G)$ and $\epsilon(G)-s(G)$. One curious observation is that $\epsilon(G)-s(G)=0$ for many families of graphs. For example, $\epsilon(K_n)=2n-3$ \cite{mprsst} and $s(K_n)=2n-3$ \cite{hw}. We showed in Theorem~\ref{thm:esnbound} that there exists a graph $G$ such that $\epsilon(G)-s(G)\geq 1$; however, we do not know if $\epsilon(G)-s(G)$ could be arbitrarily large. Nonetheless, we show in the following theorem that $s(G)-\chi'(G)$ can get arbitrarily large.

\begin{theorem}\label{thm:disjointtriangles}
Let $G$ be a disjoint union of $n$ triangles, with vertex set and edge set $V(G)=\{u_i,v_i,w_i:1\leq i\leq n\}$ and $E(G)=\{u_iv_i,v_iw_i,w_iu_i:1\leq i\leq n\}$, respectively. Let $k$ be the minimum positive integer such that $\binom{k}{3}\geq n$. Then $s(G)=k$, and $s(G)-\chi'(G)=\Theta(n^{1/3})$.
\end{theorem}

\begin{proof}
Let $f$ be a sum index labeling of $G$. If there exists $1\leq i,j\leq n$ such that
$$\{f^+(u_iv_i),f^+(v_iw_i),f^+(w_iu_i)\}=\{f^+(u_jv_j),f^+(v_jw_j),f^+(w_ju_j)\}=\{\alpha,\beta,\gamma\},$$
then $\{f(u_i),f(v_i),f(w_i)\}=\{f(u_j),f(v_j),f(w_j)\}$ by observing that the matrix equation
$$\begin{pmatrix}
1&1&0\\
0&1&1\\
1&0&1\end{pmatrix}\begin{pmatrix}
a\\
b\\
c\end{pmatrix}=\begin{pmatrix}
\alpha\\
\beta\\
\gamma\end{pmatrix}$$
has a unique solution to unknowns $a$, $b$, and $c$. By the injectivity of $f$, we have $i=j$. Hence, each triangle in $G$ must use a distinct $3$-subset of $f^+(E(G))$ for its edge labels, implying that $\binom{s(G)}{3}=\binom{|f^+|}{3}\geq n$. Therefore, $s(G)\geq k$.

To show that $s(G)\leq k$, let $\{\{\alpha_i,\beta_i,\gamma_i\}:1\leq i\leq n\}$
be a set of distinct $3$-subsets of $\{4^1,4^2,\dotsc,4^k\}$. Let $g:V(G)\to\Z$ such that $g(u_i)=(\alpha_i-\beta_i+\gamma_i)/2$, $g(v_i)=(\alpha_i+\beta_i-\gamma_i)/2$, and $g(w_i)=(-\alpha_i+\beta_i+\gamma_i)/2$. We shall verify that $g$ is injective. Note that $g(u_i)=g(v_i)$ implies $\beta_i=\gamma_i$, contradicting that $\{\alpha_i,\beta_i,\gamma_i\}$ forms a $3$-subset. Thus, $g(u_i)$ is distinct from $g(v_i)$. Similar argument shows that $g(u_i)$, $g(v_i)$, and $g(w_i)$ are mutually distinct. If $g(u_i)=g(u_j)$ for some $i\neq j$, then
$$\alpha_i+\beta_j+\gamma_i=\alpha_j+\beta_i+\gamma_j.$$
Viewing each side of the equality as the quaternary expansion of a positive integer, we observe that the equality holds if and only if $\{\alpha_i,\beta_j,\gamma_i\}=\{\alpha_j,\beta_i,\gamma_j\}$. Since $\beta_j$ is distinct from $\alpha_j$ and $\gamma_j$, we have $\beta_j=\beta_i$, which in turn implies $\{\alpha_i,\gamma_i\}=\{\alpha_j,\gamma_j\}$. As a result, we have $\{\alpha_i,\beta_i,\gamma_i\}=\{\alpha_j,\beta_j,\gamma_j\}$, contradicting that $\{\alpha_i,\beta_i,\gamma_i\}$ are distinct $3$-subsets. Hence, $g(u_i)\neq g(u_j)$. Similar argument establishes that $g(u_i)\neq g(v_j)$, $g(u_i)\neq g(w_j)$, and $g(v_i)\neq g(w_j)$. Therefore, $g$ is injective, thus $g$ is a vertex labeling. Consequently, $s(G)\leq|g^+|=|\{4^1,4^2,\dotsc,4^k\}|=k$.

It is trivial to see that $\chi'(G)=3$. Hence, $s(G)-\chi'(G)=k-3=\Theta(n^{1/3})$, since $k$ is the minimum positive integer such that $\frac{k(k-1)(k-2)}{6}\geq n$.
\end{proof}

\subsection{Sum index of cycles, spiders, wheels, and rectangular grids}\label{subsec:sumfamilies}

In this subsection, we extend the investigation of Harrington and Wong by determining the sum index for several families of graphs. First, we develop a lemma that allows us to shift vertex labels.

\begin{lemma}\label{lem:shifting}
Let $f$ be a vertex labeling of $G$. Let $g$ be a vertex labeling of $G$ such that for all vertices $v\in V(G)$, $g(v)=f(v)+c$ for some integer $c$. Then $|f^+|=|g^+|$.
\end{lemma}
\begin{proof}
For each edge $uv$ of $G$,
$$g^+(uv)=g(u)+g(v)=f(u)+f(v)+2c=f^+(uv)+2c.$$
This induces a well-defined bijection $h:f^+(E(G))\to g^+(E(G))$ such that $h(x)=x+2c$, and hence, $|f^+| = |g^+|$.
\end{proof}

The following corollary is an immediate consequence of Lemma~\ref{lem:shifting}.
\begin{corollary}\label{cor:0vertex}
Let $v$ be a vertex of $G$. There exists a sum index labeling $f$ such that $f(v) = 0$.
\end{corollary}

In the rest of this subsection, we determine the sum index of cycles, spiders, wheels, and $d$-dimensional rectangular grids. We also look into the bounds of the sum index of prisms and triangular grids. We begin with cycles.

\begin{theorem}\label{thm:sumcycles}
Let $n\geq3$ be an integer. Then $s(C_n)=3$.
\end{theorem}

\begin{proof}
Let $C_n$ be the cycle $v_0v_1v_2\dotsb v_{n-1}v_0$. Define $f$ to be a vertex labeling of $C_n$ such that $f(v_i)=(-1)^ii$ for all $0\leq i\leq n-1$. Then $f^+(v_iv_{i+1})=f(v_i)+f(v_{i+1})=\pm1$ for all $0\leq i\leq n-2$, and $f^+(v_{n-1}v_0)=f(v_{n-1})+f(v_0)=(-1)^{n-1}(n-1)\neq\pm1$. As a result, $|f^+|=3$, so $s(C_n)\leq3$. It remains to prove that $s(C_n)\geq3$.

By Proposition $\ref{thm:chromindex}$, $s(C_n)\geq\chi'(C_n)=3$ if $n$ is odd, and $s(C_n)\geq\chi'(C_n)=2$ if $n$ is even. If $n$ is even, for the sake of contradiction, assume that $g$ is a sum index labeling of $C_n$ such that the image of $g^+$ is the set $\{\alpha,\beta\}$, where $\alpha\neq\beta$. Since $g^+$ is a proper edge coloring of $C_n$, without loss of generality, assume that $g^+(v_iv_{i+1})=\alpha$ if $i$ is even and $g^+(v_iv_{i+1})=\beta$ if $i$ is odd, where addition in the indices is performed modulo $n$. As a result,
$$\sum_{i=0}^{n-1}g(v_i)=\sum_{\substack{0\leq i\leq n-1\\i\text{ is even}}}g^+(v_iv_{i+1})=\frac{n}{2}\cdot\alpha$$
and $$\sum_{i=0}^{n-1}g(v_i)=\sum_{\substack{0\leq i\leq n-1\\i\text{ is odd}}}g^+(v_iv_{i+1})=\frac{n}{2}\cdot\beta,$$
contradicting that $\alpha\neq\beta$. Therefore, $s(C_n)\geq3$ for all integers $n\geq3$.
\end{proof}

Let $\Delta\geq3$ be an integer, and let $\ell_1,\ell_2,\dotsc,\ell_\Delta$ be positive integers. A \emph{spider} $S_{\ell_1,\ell_2,\dotsc,\ell_\Delta}$ is a graph that consists of a center vertex $v_0$, which serves as a common endpoint of $\Delta$ paths $v_0v_{i,1}v_{i,2}\dotsb v_{i,\ell_i}$, where $1\leq i\leq\Delta$.

\begin{theorem}\label{thm:spider}
For every spider $S_{\ell_1,\ell_2,\dotsc,\ell_\Delta}$, $s(S_{\ell_1,\ell_2,\dotsc,\ell_\Delta})=\Delta$.
\end{theorem}

\begin{proof}
Since a lower bound of the sum index is the maximum degree, we have $s(S_{\ell_1,\ell_2,\dotsc,\ell_\Delta})\geq\Delta$. Hence, it suffices to find a vertex labeling $f$ such that $|f^+| = \Delta$.

Let $\xi$ be the smallest nonnegative integer such that $\Delta\mod{\xi}{2}$, and let $\alpha=\frac{\Delta+\xi}{2}$. Define $f(v_0)=0$. For all $1\leq i\leq\Delta$ and $1\leq j\leq\ell_i$, define
$$f(v_{i,j})=(-1)^{\Delta-i+j-1}\left((j-1)\alpha+\left\lceil\frac{i+\xi}{2}\right\rceil\right).$$
Figure~\ref{fig:spider} shows the vertex labeling of $f$ on the spider $S_{3,1,2,4,2,3,4}$.
\begin{figure}[H]
\centering
\begin{tikzpicture}[scale=1.5]
\foreach\r in{0.1/1.5}{\filldraw(0,0)circle(\r);\foreach\i in{1,2,3}{\filldraw({\i*cos(0*360/7)},{\i*sin(0*360/7)})circle(\r);}\foreach\i in{1}{\filldraw({\i*cos(1*360/7)},{\i*sin(1*360/7)})circle(\r);}\foreach\i in{1,2}{\filldraw({\i*cos(2*360/7)},{\i*sin(2*360/7)})circle(\r);}\foreach\i in{1,2,3,4}{\filldraw({\i*cos(3*360/7)},{\i*sin(3*360/7)})circle(\r);}\foreach\i in{1,2}{\filldraw({\i*cos(4*360/7)},{\i*sin(4*360/7)})circle(\r);}\foreach\i in{1,2,3}{\filldraw({\i*cos(5*360/7)},{\i*sin(5*360/7)})circle(\r);}\foreach\i in{1,2,3,4}{\filldraw({\i*cos(6*360/7)},{\i*sin(6*360/7)})circle(\r);}}\foreach[count=\i]\j in{3,1,2,4,2,3,4}{\draw(0,0)--({\j*cos((\i-1)*360/7)},{\j*sin((\i-1)*360/7)});}
\node[left]at(-0.1,0){$0$};\node[above right]at({1*cos(0*360/7)},{1*sin(0*360/7)}){$1$};\node[above right]at({2*cos(0*360/7)},{2*sin(0*360/7)}){$-5$};\node[above right]at({3*cos(0*360/7)},{3*sin(0*360/7)}){$9$};\node[above right]at({1*cos(1*360/7)},{1*sin(1*360/7)}){$-2$};\node[above right]at({1*cos(2*360/7)},{1*sin(2*360/7)}){$2$};\node[above right]at({2*cos(2*360/7)},{2*sin(2*360/7)}){$-6$};\node[above right]at({1*cos(3*360/7)},{1*sin(3*360/7)}){$-3$};\node[above right]at({2*cos(3*360/7)},{2*sin(3*360/7)}){$7$};\node[above right]at({3*cos(3*360/7)},{3*sin(3*360/7)}){$-11$};\node[above right]at({4*cos(3*360/7)},{4*sin(3*360/7)}){$15$};\node[above left]at({1*cos(4*360/7)},{1*sin(4*360/7)}){$3$};\node[above left]at({2*cos(4*360/7)},{2*sin(4*360/7)}){$-7$};\node[above left]at({1*cos(5*360/7)},{1*sin(5*360/7)}){$-4$};\node[above left]at({2*cos(5*360/7)},{2*sin(5*360/7)}){$8$};\node[above left]at({3*cos(5*360/7)},{3*sin(5*360/7)}){$-12$};\node[above right]at({1*cos(6*360/7)},{1*sin(6*360/7)}){$4$};\node[above right]at({2*cos(6*360/7)},{2*sin(6*360/7)}){$-8$};\node[above right]at({3*cos(6*360/7)},{3*sin(6*360/7)}){$12$};\node[above right]at({4*cos(6*360/7)},{4*sin(6*360/7)}){$-16$};
\end{tikzpicture}
\caption{The vertex labeling $f$ on the spider $S_{3,1,2,4,2,3,4}$}
\label{fig:spider}
\end{figure}

To verify that $f$ is a vertex labeling, we need to prove that $f$ is injective. Note that
$$(j-1)\alpha<|f(v_{i,j})|=(j-1)\alpha+\left\lceil\frac{i+\xi}{2}\right\rceil\leq j\alpha,$$
so $f(v_{i,j})\neq f(v_{i',j'})$ if $j\neq j'$. Moreover, if $j=j'$, then $f(v_{i,j})\neq f(v_{i',j})$ since $(-1)^{-i}\left\lceil\frac{i+\xi}{2}\right\rceil$ are all distinct. To verify that $|f^+|=\Delta$, we show that
$$
f^+(E(S_{\ell_1,\ell_2,\dotsc,\ell_\Delta}))=\{f^+(v_0v_{i,1}):1\leq i\leq\Delta\},$$
which has cardinality $\Delta$ since $f^+$ is a proper edge coloring. Note that $\{f^+(v_0v_{i,1}):1\leq i\leq\Delta\}$ is $\{-1,1,-2,2,\dotsc,-\alpha,\alpha\}$ and $\{1,-2,2,\dotsc,-\alpha,\alpha\}$ when $\Delta$ is even and odd, respectively. For each $1\leq i\leq\Delta$ and $1\leq j\leq\ell_i-1$,
\begin{align*}
    f^+(v_{i,j}v_{i,j+1})&=(-1)^{\Delta-i+j-1}\left((j-1)\alpha+\left\lceil\frac{i+\xi}{2}\right\rceil\right)+(-1)^{\Delta-i+j}\left(j\alpha+\left\lceil\frac{i+\xi}{2}\right\rceil\right)\\
    &=(-1)^{\Delta-i+j}\alpha,
\end{align*}
which is an element of $\{f^+(v_0v_{i,1}):1\leq i\leq\Delta\}$.
\end{proof}

Let $\Delta\geq3$ be an integer. A \emph{wheel} $W_\Delta$ is a graph that consists of a center vertex $v_0$, a cycle $v_1v_2\dotsb v_nv_1$, and edges $v_0v_i$ for all $1\leq i\leq n$.

\begin{theorem}\label{thm:wheel}
Let $\Delta\geq3$ be an integer, and let $W_{\Delta}$ be the wheel graph with maximum degree $\Delta$. Then $s(W_\Delta)=\max\{5,\Delta\}$.
\end{theorem}

\begin{proof}
Since the maximum degree serves as a lower bound for $s(W_\Delta)$, we have $s(W_\Delta)\geq\Delta$. Tuga and Miller showed that if $\Delta\geq5$, then $\epsilon(W_\Delta)=\Delta$ \cite{ESLradius1}. By Theorem~\ref{thm:esnbound}, we have $s(W_\Delta)\leq\Delta$. Therefore, when $\Delta\geq5$, we have $s(W_\Delta)=\Delta$.

Note that $W_3$ is isomorphic to the complete graph $K_4$. By Harrington and Wong, we have $s(W_3)=s(K_4)=2\cdot4-3=5$. It remains to show that $s(W_4)=5$. For the sake of contradiction, assume that $f$ is a sum index labeling of $W_4$, where the image of $f^+$ is the set $\{\alpha,\beta,\gamma,\delta\}$ and $\alpha,\beta,\gamma,\delta$ are distinct. By Corollary~\ref{cor:0vertex}, we may further assume that $f(v_0)=0$. Since $f^+$ forms a proper edge labeling on $W_4$, it is not difficult to see that the only two candidates for $f^+$ are given by Figures~\ref{fig:wheel1} and \ref{fig:wheel2}.
\begin{figure}[H]
\centering
\begin{minipage}{0.45\linewidth}
\centering
\begin{tikzpicture}
\coordinate(v0)at(0,0);\coordinate(v1)at(2,0);\coordinate(v2)at(0,2);\coordinate(v3)at(-2,0);\coordinate(v4)at(0,-2);\foreach\i in{v0,v1,v2,v3,v4}{\filldraw(\i)circle(0.1);}\draw(v3)--(v1)--(v2)--(v3)--(v4)--(v1);\draw(v2)--(v4);\node[below right]at(v0){$v_0$};\node[above right]at(v1){$v_1$};\node[above right]at(v2){$v_2$};\node[above left]at(v3){$v_3$};\node[below right]at(v4){$v_4$};\node[above]at(1,0){$\alpha$};\node[left]at(0,1){$\beta$};\node[above]at(-1,0){$\gamma$};\node[left]at(0,-1){$\delta$};\node[above right]at(1,1){$\gamma$};\node[above left]at(-1,1){$\delta$};\node[below left]at(-1,-1){$\alpha$};\node[below right]at(1,-1){$\beta$};
\end{tikzpicture}
\caption{Candidate $1$ for $f^+$ on $W_4$}
\label{fig:wheel1}
\end{minipage}\qquad
\begin{minipage}{0.45\linewidth}
\centering
\begin{tikzpicture}
\coordinate(v0)at(0,0);\coordinate(v1)at(2,0);\coordinate(v2)at(0,2);\coordinate(v3)at(-2,0);\coordinate(v4)at(0,-2);\foreach\i in{v0,v1,v2,v3,v4}{\filldraw(\i)circle(0.1);}\draw(v3)--(v1)--(v2)--(v3)--(v4)--(v1);\draw(v2)--(v4);\node[below right]at(v0){$v_0$};\node[above right]at(v1){$v_1$};\node[above right]at(v2){$v_2$};\node[above left]at(v3){$v_3$};\node[below right]at(v4){$v_4$};\node[above]at(1,0){$\alpha$};\node[left]at(0,1){$\beta$};\node[above]at(-1,0){$\gamma$};\node[left]at(0,-1){$\delta$};\node[above right]at(1,1){$\delta$};\node[above left]at(-1,1){$\alpha$};\node[below left]at(-1,-1){$\beta$};\node[below right]at(1,-1){$\gamma$};
\end{tikzpicture}
\caption{Candidate $2$ for $f^+$ on $W_4$}
\label{fig:wheel2}
\end{minipage}
\end{figure}

In both Figures~\ref{fig:wheel1} and \ref{fig:wheel2}, since $f(v_0)=0$, we have $(f(v_1),f(v_2),f(v_3),f(v_4))=(\alpha,\beta,\gamma,\delta)$ from the spokes of $W_4$. In Figure~\ref{fig:wheel1}, from the cycle of $W_4$, we obtain the matrix equation
$$\begin{pmatrix}
1&1&0&0\\
0&1&1&0\\
0&0&1&1\\
1&0&0&1
\end{pmatrix}\begin{pmatrix}\alpha\\\beta\\\gamma\\\delta\end{pmatrix}=\begin{pmatrix}\gamma\\\delta\\\alpha\\\beta\end{pmatrix}.$$
It follows that
$$\begin{pmatrix}
1&1&0&0\\
0&1&1&0\\
0&0&1&1\\
1&0&0&1
\end{pmatrix}\begin{pmatrix}\alpha\\\beta\\\gamma\\\delta\end{pmatrix}=\begin{pmatrix}
0&0&1&0\\
0&0&0&1\\
1&0&0&0\\
0&1&0&0
\end{pmatrix}\begin{pmatrix}\alpha\\\beta\\\gamma\\\delta\end{pmatrix},$$
and by subtraction and factorization, we have
$$\begin{pmatrix}
1&1&-1&0\\
0&1&1&-1\\
-1&0&1&1\\
1&-1&0&1
\end{pmatrix}\begin{pmatrix}\alpha\\\beta\\\gamma\\\delta\end{pmatrix}=\begin{pmatrix}0\\0\\0\\0\end{pmatrix}.$$
Since the $4\times4$ matrix in the last equation is invertible, we arrive at the contradiction that $\alpha=\beta=\gamma=\delta=0$. A similar argument shows that the edge labeling in Figure~\ref{fig:wheel2} is also impossible.
\end{proof}

Let $2\leq m\leq n$ be integers. An \emph{$n\times m$ rectangular grid} $L_{n\times m}$ is a graph with the vertex set $\{v_{i,j}:0\leq i\leq n-1,0\leq j\leq m-1\}$ and the edge set $\{v_{i,j}v_{i',j'}:|i-i'|+|j-j'|=1\}$. A special type of rectangular grid is $L_{n\times2}$, which is a \emph{ladder graph} with $n$ ``rungs."

\begin{theorem}\label{thm:grids}
Let $G$ be a rectangular grid.  If $G$ is a ladder graph, then $s(G)=3$; otherwise, $s(G)=4$.
\end{theorem}

\begin{proof}
First, consider $m=2$. If $n=2$, then $G$ is isomorphic to $C_4$, and $s(G)=3$ by Theorem~\ref{thm:sumcycles}. If $n>2$, then $s(G)\geq\Delta(G)=3$. To show that $s(G)\leq3$, we define
$$f(v_{i,j})=\begin{cases}
-i&\text{if }i\text{ is even and }j=0;\\
i+1&\text{if }i\text{ is odd and }j=0;\\
i+1&\text{if }i\text{ is even and }j=1;\\
-i&\text{if }i\text{ is odd and }j=1.
\end{cases}$$
It is easy to verify that $f$ forms a vertex labeling on $G$ and $f^+(E(G))=\{0,1,2\}$. Hence, $s(G)\leq|f^+|=3$.

Next, consider $m>2$. Note that $s(G)\geq\Delta(G)=4$, so it remains to show that $s(G)\leq4$. Define $f(v_{i,j})=(-1)^{i+j}(mi+j)$. Viewing $mi+j$ as the base $m$ expansion of a positive integer, we see that $f$ is injective. Furthermore, $|f^+|=4$ since
$$f(v_{i,j})+f(v_{i+1,j})=(-1)^{i+j}(mi+j)+(-1)^{i+1+j}(m(i+1)+j)=\pm m$$
and
$$f(v_{i,j})+f(v_{i,j+1})=(-1)^{i+j}(mi+j)+(-1)^{i+j+1}(mi+j+1)=\pm1.$$
Hence, $s(G)\leq|f^+|=4$. Figures~\ref{fig:ladder} and \ref{fig:grid} illustrate the vertex labelings $f$ on $L_{6\times2}$ and $L_{6\times3}$, respectively.
\begin{figure}[H]
\centering
\begin{minipage}{0.45\linewidth}
\centering
\begin{tikzpicture}[scale=1.3]
\foreach\i in{0,1,2,3,4,5}{\foreach\j in{0,1}{\filldraw(\i,\j)circle(0.1/1.3);}}\draw(0,0)grid(5,1);\node[below]at(0,0){$0$};\node[below]at(1,0){$2$};\node[below]at(2,0){$-2$};\node[below]at(3,0){$4$};\node[below]at(4,0){$-4$};\node[below]at(5,0){$6$};\node[above]at(0,1){$1$};\node[above]at(1,1){$-1$};\node[above]at(2,1){$3$};\node[above]at(3,1){$-3$};\node[above]at(4,1){$5$};\node[above]at(5,1){$-5$};
\end{tikzpicture}
\caption{Vertex labeling $f$ on $L_{6\times2}$}
\label{fig:ladder}
\end{minipage}\qquad
\begin{minipage}{0.45\linewidth}
\centering
\begin{tikzpicture}[scale=1.3]
\foreach\i in{0,1,2,3,4,5}{\foreach\j in{0,1,2}{\filldraw(\i,\j)circle(0.1/1.3);}}\draw(0,0)grid(5,2);\node[below]at(0,0){$0$};\node[below]at(1,0){$-3$};\node[below]at(2,0){$6$};\node[below]at(3,0){$-9$};\node[below]at(4,0){$12$};\node[below]at(5,0){$-15$};\node[above left]at(0,1){$-1$};\node[above left]at(1,1){$4$};\node[above left]at(2,1){$-7$};\node[above left]at(3,1){$10$};\node[above left]at(4,1){$-13$};\node[above left]at(5,1){$16$};\node[above]at(0,2){$2$};\node[above]at(1,2){$-5$};\node[above]at(2,2){$8$};\node[above]at(3,2){$-11$};\node[above]at(4,2){$14$};\node[above]at(5,2){$-17$};
\end{tikzpicture}
\caption{Vertex labeling $f$ on $L_{6\times3}$}
\label{fig:grid}
\end{minipage}
\end{figure}
\end{proof}

Since a \emph{prism graph} $\Pi_n$ is a ladder graph $L_{n\times 2}$ with the two additional edges $v_{0,0}v_{n-1,0}$ and $v_{0,1}v_{n-1,1}$, the following corollary follows immediately from Theorem~\ref{thm:grids}.

\begin{corollary}\label{cor:prism}
For every prism graph $\Pi_n$, $s(\Pi_n)\leq5$.
\end{corollary}%see if we can improve it later

We extend the vertex labeling of rectangular grids given in the proof of Theorem~\ref{thm:grids} to obtain the next corollary.

\begin{corollary}\label{cor:highergrids}
Let $G$ be a $d$-dimensional rectangular grid with maximum degree $2d$. Then $s(G)=2d$.
\end{corollary}

\begin{proof}
Define the vertex set of $G$ as
$$\{v_\i:\i=(i_1,i_2,\dotsc,i_d),0\leq i_1\leq n_1-1,0\leq i_2\leq n_2-2,\dotsc,0\leq i_d\leq n_d-1\}$$
for some integers $n_1\geq n_2\geq\dotsb\geq n_d\geq3$. To extend the vertex labeling in the proof of Theorem \ref{thm:grids} to higher dimensions, let
$$f(v_\i) = (-1)^{i_1+i_2+\dotsb+i_d}(n_1^{d-1}i_1+n_1^{d-2}i_2+\dotsb+n_1i_{d-1}+i_d).$$
Again, it is easy to see that $f$ is injective by viewing the image of $f$ as the base $n_1$ expansion of an integer.  To see that $|f^+| = 2d$, note that for any two adjacent vertices $v_\i$ and $v_{\i'}$, where $\i=(i_1,i_2,\dotsc,i_d)$ and  $\i'=(i_1,i_2,\dotsc,i_j\pm1,\dotsc,i_d)$ for some $1\leq j\leq d$, 
\begin{align*}
|f(v_\i)+ f(v_{\i'})|&=|(n_1^{d-1}i_1+n_1^{d-2}i_2+\dotsb+n_1i_{d-1}+i_d)\\
&\hspace{20pt}-(n_1^{d-1}i_1+n_1^{d-2}i_2+\dotsb+n_1^{d-j}(i_j\pm1)+\dotsb+n_1i_{d-1}+i_d)|\\
&=n_1^{d-j}.
\end{align*}
This proves that $s(G)\leq2d$, and our proof is complete by noting that $s(G)\geq\Delta(G)=2d$.
\end{proof}

After studying the sum index of rectangular grids, the last result of this subsection is on triangular grids. A \emph{triangular grid} $T_n$ is a graph with the vertex set $\{v_{i,j}:0\leq i\leq n-1,0\leq j\leq i\}$ and the edge set $\{v_{i,j}v_{i',j'}:|i-i'|+|j-j'|=1\text{ or }(i-i')(j-j')=1\}$. For example, Figure~\ref{fig:triangular} shows the triangular grid $T_3$.
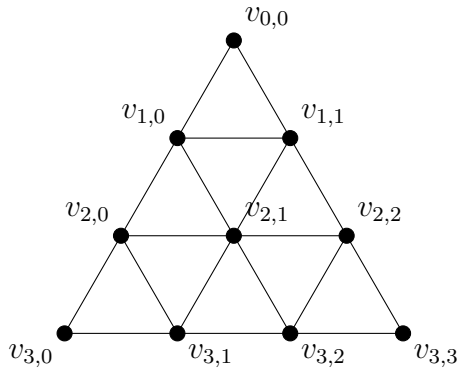
\begin{figure}[H]
\centering
\begin{tikzpicture}[scale=1.5]
\foreach\r in{0.1/1.5}{\foreach\i in{(0,0),(1,0),(2,0),(3,0),(0.5,{sin(60)}),(1.5,{sin(60)}),(2.5,{sin(60)}),(1,{2*sin(60)}),(2,{2*sin(60)}),(1.5,{3*sin(60)})}{\filldraw\i circle(\r);}}\draw(0,0)--(3,0);\draw(0.5,{sin(60)})--(2.5,{sin(60)});\draw(1,{2*sin(60)})--(2,{2*sin(60)});\draw(0,0)--(1.5,{3*sin(60)});\draw(1,0)--(2,{2*sin(60)});\draw(2,0)--(2.5,{sin(60)});\draw(3,0)--(1.5,{3*sin(60)});\draw(2,0)--(1,{2*sin(60)});\draw(1,0)--(0.5,{sin(60)});\node[above right]at(1.5,{3*sin(60)}){$v_{0,0}$};\node[above left]at(1,{2*sin(60)}){$v_{1,0}$};\node[above right]at(2,{2*sin(60)}){$v_{1,1}$};\node[above left]at(0.5,{sin(60)}){$v_{2,0}$};\node[above right]at(1.5,{sin(60)}){$v_{2,1}$};\node[above right]at(2.5,{sin(60)}){$v_{2,2}$};
\node[below left]at(0,0){$v_{3,0}$};\node[below right]at(1,0){$v_{3,1}$};\node[below right]at(2,0){$v_{3,2}$};\node[below right]at(3,0){$v_{3,3}$};
\end{tikzpicture}
\caption{Triangular grid $T_3$}
\label{fig:triangular}
\end{figure}

As we can see, there are $n$ rows of triangles in $T_n$, where the $i$-th row contains $i$ upward-facing triangles and $i-1$ downward-facing triangles. Note that in the $i$-th row, there are at least $\frac{i}{2}$ vertex-disjoint upward-facing triangles. Hence, by considering every other row of $T_n$ starting from the last row, there are at least
$$\sum_{\substack{1\leq i\leq n\\n-i\text{ is even}}}\left\lceil\frac{i}{2}\right\rceil=1+2+\dotsb+\left\lceil\frac{n}{2}\right\rceil=\frac{\left\lceil\frac{n}{2}\right\rceil\left(\left\lceil\frac{n}{2}\right\rceil+1\right)}{2}\geq\frac{n^2}{8}$$
vertex-disjoint triangles in $T_n$. By Theorem~\ref{thm:disjointtriangles}, we have the following corollary.

\begin{corollary}
The sum index of triangular grids grows with $n$. To be more precise, $s(T_n)=\Omega(n^{2/3})$.
\end{corollary}

\subsection{Constructing graphs with a prescribed sum index}\label{subsec:construct}

Now that we have found the sum index of various families of graphs, we ask the converse question: if we know the sum index of a graph, what can we say about the graph? For certain values $k$, we can characterize all graphs of sum index $k$.  It should be noted that isolated vertices do not affect the sum index, so we will ignore them for the purpose of this subsection.

\begin{theorem}\label{thm:s=1,2}
\begin{enumerate}[$(a)$]
\item\label{item:s=1} If $s(G)=1$, then $G$ is a disjoint union of copies of $K_2$.
\item\label{item:s=2} If $s(G)=2$, then $G$ is a disjoint union of paths.
\end{enumerate}
\end{theorem}

\begin{proof}\begin{itemize}
\item[$(\ref{item:s=1})$] This statement follows from the fact that $s(G)\geq\Delta(G)$.
\item[$(\ref{item:s=2})$] Again, from the fact that $s(G)\geq\Delta(G)$, if $s(G)=2$, then $G$ is a disjoint union of cycles and paths. However, Theorem~\ref{thm:sumcycles} implies that any graph that contains a cycle has sum index at least $3$. Therefore, $G$ is a disjoint union of paths.
\end{itemize}
\end{proof}

\begin{theorem}\label{thm:s=2n-3,2n-4}
\label{sum index of almost complete}
Let $G$ be a graph with $n$ vertices.
\begin{enumerate}[$(a)$]
\item\label{item:s=2n-3} The sum index $s(G)=2n-3$ if and only if $G=K_n$.
\item\label{item:s=2n-4} The sum index $s(G)=2n-4$ if and only if $G=K_n\setminus\{e\}$ for some edge $e$.
\end{enumerate}
\end{theorem}

\begin{proof}
\begin{itemize}
\item[$(\ref{item:s=2n-3})$] As mentioned in Section~\ref{sec:intro}, $s(K_n)=2n-3$. If $G\neq K_n$, then let $v_1$ and $v_2$ be two nonadjacent vertices in $G$. Define a vertex labeling $f:V(G)\to\{1,2,\dotsc,n\}$ such that $f(v_1)=1$ and $f(v_2)=2$. Then $f^+(E(G))\subseteq\{4,5,\dotsc,2n-1\}$, and $|f^+|\leq2n-4$.  Therefore, $s(G)\leq 2n-4$, contradicting that $s(G)=2n-3$.

\item[$(\ref{item:s=2n-4})$] If $G=K_n\setminus\{e\}$, then $s(G)\leq2n-4$ from the proof of part $(\ref{item:s=2n-3})$. Since $s(K_n)=2n-3$, every vertex labeling $f$ on $K_n$ satisfies $|f^+(E(K_n))|\geq2n-3$, thus $|f^+(E(K_n\setminus\{e\}))|\geq2n-4$. Therefore, $s(G)=2n-4$.

If $s(G)=2n-4$, then part $(\ref{item:s=2n-3})$ implies that $G$ is a proper subgraph of $K_n$. For the sake of contradiction, assume that $G$ is a subgraph of $K_n\setminus\{e_1,e_2\}$ for two distinct edges $e_1$ and $e_2$.

If $e_1$ and $e_2$ share a common vertex, then let $e_1=v_1v_2$ and $e_2=v_1v_3$. Define a vertex labeling $f:V(G)\to\{1,2,\dotsc,n\}$ such that $f(v_1)=1$, $f(v_2)=2$, and $f(v_3)=3$. As a result, $f^+(E(G))\subseteq\{5,6,\dotsc,2n-1\}$, and $|f^+|\leq2n-5$. Similarly, if $e_1$ and $e_2$ do not share a common vertex, then let $e_1=v_1v_2$ and $e_2=v_3v_4$. Define a vertex labeling $f:V(G)\to\{1,2,\dotsc,n\}$ such that $f(v_1)=1$, $f(v_2)=2$, $f(v_3)=n-1$, and $f(v_4)=n$. As a result, $f^+(E(G))\subseteq\{4,5,\dotsc,2n-2\}$, and $|f^+|\leq2n-5$. In both cases, $s(G)\leq2n-5$, contradicting that $s(G)=2n-4$. 
\end{itemize}
\end{proof}

Using Theorem~\ref{thm:s=2n-3,2n-4}, we prove that any sum index between $2$ and $2n-3$ is attainable by a connected graph $G$ with $n$ vertices.

\begin{theorem}
Let $n$ and $k$ be positive integers such that $2\leq k\leq 2n-3$. Then there exists a connected graph $G$ with $n$ vertices such that $s(G)=k$.
\end{theorem}

\begin{proof}
If $k=2$, then we are done by considering $G$ to be a path on $n$ vertices. For the rest of this proof, we assume that $k\geq3$.

Let the vertex set of $G$ be $V(G)=\{v_1,v_2,\dotsc,v_n\}$. If $k$ is odd, then let $\ell$ be an integer such that $2\ell-3=k$, and let the edge set of $G$ be
$$E(G)=\{v_iv_j:1\leq i<j\leq\ell\}\cup\{v_1v_{\ell+1},v_iv_{i+1}:\ell+1\leq i\leq n-1\};$$
if $k$ is even, then let $\ell$ be an integer such that $2\ell-4=k$, and let the edge set of $G$ be
$$E(G)=\{v_iv_j:1\leq i<j\leq\ell\text{ and }i<\ell-1\}\cup\{v_1v_{\ell+1},v_iv_{i+1}:\ell+1\leq i\leq n-1\}.$$

By Theorem~\ref{thm:s=2n-3,2n-4}, if $k$ is odd, then $s(G)\geq2\ell-3=k$; if $k$ is even, then $s(G)\geq2\ell-4=k$. It remains to show that $s(G)\leq k$ by defining a vertex labeling $f:V(G)\to\Z$ such that $|f^+|=k$.

Let $f:V(G)\to\Z$ such that $f(v_i)=i$ for $1\leq i\leq\ell$ and
$$f(v_i)=\begin{cases}
\ell+\left\lceil\frac{i-\ell}{2}\right\rceil&\text{if $i-\ell$ is odd};\\
-\frac{i-\ell}{2}+1&\text{if $i-\ell$ is even}
\end{cases}$$
for $\ell+1\leq i\leq n$. %The following two figures illustrate the vertex labelings $f$ on $G$ with $12$ vertices: Figure~\ref{fig:n=12s=7} corresponds to the case when $k=7$, and Figure~\ref{fig:n=12s=8} corresponds to the case when $k=8$.
Figures~\ref{fig:n=12s=7} and \ref{fig:n=12s=8} illustrate the vertex labelings $f$ on $G$ with $12$ vertices for $k=7$ and $k=8$, respectively. 
\begin{figure}[H]
\centering
\begin{minipage}{0.45\linewidth}
\centering
\begin{tikzpicture}[scale=1.3]
\foreach\i in{0,1,...,4}{\filldraw({cos(\i*72)},{sin(\i*72)})circle(0.1/1.3);}\foreach\i in{1,3,...,7}{\filldraw(1+0.5*\i,0.8)circle(0.1/1.3);}\foreach\i in{1,3,5}{\filldraw(1.5+0.5*\i,-0.8)circle(0.1/1.3);}\draw({cos(0*72)},{sin(0*72)})--({cos(1*72)},{sin(1*72)})--({cos(2*72)},{sin(2*72)})--({cos(3*72)},{sin(3*72)})--({cos(4*72)},{sin(4*72)})--({cos(0*72)},{sin(0*72)})--({cos(2*72)},{sin(2*72)})--({cos(4*72)},{sin(4*72)})--({cos(1*72)},{sin(1*72)})--({cos(3*72)},{sin(3*72)})--({cos(0*72)},{sin(0*72)})--(1.5,0.8)--(2,-0.8)--(2.5,0.8)--(3,-0.8)--(3.5,0.8)--(4,-0.8)--(4.5,0.8);\node[below right]at({cos(0*72)},{sin(0*72)}){$1$};\node[above left]at({cos(1*72)},{sin(1*72)}){$2$};\node[above left]at({cos(2*72)},{sin(2*72)}){$3$};\node[below left]at({cos(3*72)},{sin(3*72)}){$4$};\node[below left]at({cos(4*72)},{sin(4*72)}){$5$};\node[above right]at(1.5,0.8){$6$};\node[below right]at(2,-0.8){$0$};\node[above right]at(2.5,0.8){$7$};\node[below right]at(3,-0.8){$-1$};\node[above right]at(3.5,0.8){$8$};\node[below right]at(4,-0.8){$-2$};\node[above right]at(4.5,0.8){$9$};
\end{tikzpicture}
\caption{Vertex labeling $f$ on $G$ with $12$ vertices for $k=7$}
\label{fig:n=12s=7}
\end{minipage}\qquad
\begin{minipage}{0.45\linewidth}
\centering
\begin{tikzpicture}[scale=1.3]
\foreach\i in{0,1,...,4,5}{\filldraw({cos(\i*60)},{sin(\i*60)})circle(0.1/1.3);}\foreach\i in{1,3,5}{\filldraw(1+0.5*\i,0.8)circle(0.1/1.3);}\foreach\i in{1,3,5}{\filldraw(1.5+0.5*\i,-0.8)circle(0.1/1.3);}\draw({cos(0*60)},{sin(0*60)})--({cos(1*60)},{sin(1*60)})--({cos(2*60)},{sin(2*60)})--({cos(3*60)},{sin(3*60)})--({cos(4*60)},{sin(4*60)});\draw({cos(5*60)},{sin(5*60)})--({cos(1*60)},{sin(1*60)})--({cos(3*60)},{sin(3*60)})--({cos(5*60)},{sin(5*60)})--({cos(2*60)},{sin(2*60)})--({cos(4*60)},{sin(4*60)})--({cos(0*60)},{sin(0*60)})--({cos(2*60)},{sin(2*60)});\draw({cos(5*60)},{sin(5*60)})--({cos(0*60)},{sin(0*60)})--({cos(3*60)},{sin(3*60)});\draw({cos(1*60)},{sin(1*60)})--({cos(4*60)},{sin(4*60)});\draw({cos(0*60)},{sin(0*60)})--(1.5,0.8)--(2,-0.8)--(2.5,0.8)--(3,-0.8)--(3.5,0.8)--(4,-0.8);\node[below right]at({cos(0*60)},{sin(0*60)}){$1$};\node[above right]at({cos(1*60)},{sin(1*60)}){$2$};\node[above left]at({cos(2*60)},{sin(2*60)}){$3$};\node[above left]at({cos(3*60)},{sin(3*60)}){$4$};\node[below left]at({cos(4*60)},{sin(4*60)}){$5$};\node[below right]at({cos(5*60)},{sin(5*60)}){$6$};\node[above right]at(1.5,0.8){$7$};\node[below right]at(2,-0.8){$0$};\node[above right]at(2.5,0.8){$8$};\node[below right]at(3,-0.8){$-1$};\node[above right]at(3.5,0.8){$9$};\node[below right]at(4,-0.8){$-2$};
\end{tikzpicture}
\caption{Vertex labeling $f$ on $G$ with $12$ vertices for $k=8$}
\label{fig:n=12s=8}
\end{minipage}
\end{figure}
%picture
By construction, if $k$ is odd, then $f^+(\{v_iv_j:1\leq i<j\leq\ell\})=\{3,4,\dotsc,2\ell-1\}$; if $k$ is even, then $f^+(\{v_iv_j:1\leq i<j\leq\ell\text{ and }i<\ell-1\})=\{3,4,\dotsc,2\ell-2\}$. Furthermore, regardless whether $k$ is odd or even, $f^+(\{v_1v_{\ell+1},v_iv_{i+1}:\ell+1\leq i\leq n-1\})=\{\ell+1,\ell+2\}$. Note that $\ell+2\leq2\ell-1$ if $k\geq3$ is odd, and $\ell+2\leq2\ell-2$ if $k\geq3$ is even. Therefore, if $k$ is odd, $|f^+|=2\ell-3=k$, and if $k$ is even, $|f^+|=2\ell-4=k$.
\end{proof}

\subsection{Hyperdiamonds and the sum index of trees}\label{subsec:sumtrees}

In this section, we aim to study the sum index of trees.  In particular, we will show that there exists a graph $H_k$ so that every tree with sum index less than or equal to $k$ is a subgraph of $H_k$.  The graph $H_k$, referred to as a hyperdiamond, was introduced by Miller, Ryan, and Ryj\'a\v{c}ek \cite{mrr}.

\begin{definition}
Let $k$ be a positive integer. Let $\e_1,\e_2,\dotsc,\e_k$ be the standard basis vectors of $\Z^k$. For $i=1,2,\dotsc,k$, we define the map $\psi_i:\Z^k\to\Z^k$ by $\psi_i(\x)=\e_i-\x$. The \emph{hyperdiamond group} $\Gamma_k$ is the group generated by the maps $\psi_1,\psi_2,\dotsc,\psi_k$ under composition. The \emph{hyperdiamond graph} $H_k$ is the Cayley graph of $\Gamma_k$ with generating set $S=\{\psi_1,\psi_2,\dotsc,\psi_k\}$. In other words, the vertex set of $H_k$ is $\Gamma_k$, and for any $\phi,\phi'\in \Gamma_k$, there is a directed edge from $\phi$ to $\phi'$ if and only if $\phi\circ(\phi')^{-1}\in S$. Note that $S$ is closed under inverses since $\psi_i^{-1}=\psi_i$ for all $i=1,2,\dotsc,k$, so the Cayley graph $H_k$ is an undirected graph.
\end{definition}

The following lemma provides an easy way to compute the distance between two vertices $\phi$ and $\phi'$ in $H_k$.

\begin{lemma}\label{lem:phiphi'}
Let $\phi,\phi'\in \Gamma_k$, and let $\bz=(0,0,\dotsc,0)\in\Z^k$.
\begin{enumerate}[$(a)$]
\item\label{item:phi'=0} If $\phi'(\bz)=\bz$, then the distance between $\phi$ and $\phi'$ in $H_k$ is $\lVert\phi(\bz)\rVert$. Here, and throughout this paper, $\lVert\x\rVert=|x_1|+|x_2|+\dotsb+|x_k|$ for all $\x=(x_1,x_2,\dotsc,x_k)\in\Z^k$, which is the $1$-norm on $\Z^k$.
\item\label{item:phi=phi'} If $\phi(\bz)=\phi'(\bz)$, then $\phi=\phi'$.
\end{enumerate}
\end{lemma}

\begin{proof}
Let $r$ be the distance between $\phi$ and $\phi'$ in $H_k$. In other words,
\begin{equation}\label{eqn:phi}
\phi=\psi_{i_r}\circ\psi_{i_{r-1}}\circ\dotsb\circ\psi_{i_1}\circ\phi'
\end{equation}
for some $\psi_{i_1},\psi_{i_2},\dotsc,\psi_{i_r}\in S$. Hence, 
\begin{equation}\label{eqn:phi0}
\phi(\bz)=\e_{i_r}-\e_{i_{r-1}}+\e_{i_{r-2}}-\dotsb+(-1)^{r-1}\e_{i_1}+(-1)^r\phi'(\bz).
\end{equation}

\begin{itemize}
\item[$(\ref{item:phi'=0})$] If $\phi'(\bz)=\bz$, then by the triangle inequality of the $1$-norm on $\Z^k$, 
$$\lVert\phi(\bz)\rVert\leq\lVert\e_{i_r}\rVert+\lVert\e_{i_{r-1}}\rVert+\lVert\e_{i_{r-2}}\rVert+\dotsb+\lVert\e_{i_1}\rVert=r.$$
It is easy to see that $\lVert\phi(\bz)\rVert<r$ only if there exist integers $\ell<m$ of opposite parity such that $\e_{i_\ell}=\e_{i_m}$. In that case, by cancelling $\e_{i_\ell}$ and $\e_{i_m}$ and reversing the order of the terms between $\e_{i_\ell}$ and $\e_{i_m}$, we have
\begin{align*}
\phi(\bz)&=\e_{i_r}-\e_{i_{r-1}}+\e_{i_{r-2}}-\dotsb+(-1)^{r-m-1}\e_{i_{m+1}}\\
&\hspace{20pt}+(-1)^{r-m+1}\e_{i_{m-1}}+(-1)^{r-m+2}\e_{i_{m-2}}+\dotsb+(-1)^{r-\ell-1}\e_{i_{\ell+1}}\\
&\hspace{20pt}+(-1)^{r-\ell+1}\e_{i_{\ell-1}}+(-1)^{r-\ell+2}\e_{i_{\ell-2}}+\dotsb+(-1)^{r-1}\e_{i_1}+(-1)^r\phi'(\bz)\\
&=\e_{i_r}-\e_{i_{r-1}}+\e_{i_{r-2}}-\dotsb+(-1)^{r-m-1}\e_{i_{m+1}}\\
&\hspace{20pt}+\big((-1)^{r-\ell-1}\e_{i_{\ell+1}}+(-1)^{r-\ell-2}\e_{i_{\ell+2}}+\dotsb+(-1)^{r-m+1}\e_{i_{m-1}}\big)\\
&\hspace{20pt}+(-1)^{r-\ell+1}\e_{i_{\ell-1}}+(-1)^{r-\ell+2}\e_{i_{\ell-2}}+(-1)^{r-1}\e_{i_1}+(-1)^r\phi'(\bz)\\
&=\psi_{i_r}\circ\psi_{i_{r-1}}\circ\dotsb\circ\psi_{i_{m+1}}\circ(\psi_{i_{\ell+1}}\circ\psi_{i_{\ell+2}}\circ\dotsb\circ\psi_{i_{m-1}})\\
&\hspace{22pt}\circ\psi_{i_{\ell-1}}\circ\psi_{i_{\ell-2}}\circ\dotsb\circ\psi_{i_1}\circ\phi'(\bz).
\end{align*}
This shows that the distance between $\phi$ and $\phi'$ is at most $r-2$, contradicting the assumption that $r$ is the distance between $\phi$ and $\phi'$. Therefore, $\lVert\phi(\bz)\rVert=r$.
\item[$(\ref{item:phi=phi'})$] If $r$ is odd, then equation $\eqref{eqn:phi0}$ becomes
$$2\phi(\bz)=\e_{i_r}-\e_{i_{r-1}}+\e_{i_{r-2}}-\dotsb+(-1)^{r-1}\e_{i_1}.$$
Note that $\lVert2\phi(0)\rVert$ is even and $\lVert\e_{i_r}-\e_{i_{r-1}}+\e_{i_{r-2}}-\dotsb+(-1)^{r-1}\e_{i_1}\rVert$ is odd, which is a contradiction. If $r$ is even, then equation $\eqref{eqn:phi0}$ becomes
$$\bz=\e_{i_r}-\e_{i_{r-1}}+\e_{i_{r-2}}-\dotsb+(-1)^{r-1}\e_{i_1}.$$
Hence, for all $\x\in\Z^k$, equation $\eqref{eqn:phi}$ yields 
\begin{align*}
\phi(\x)&=\psi_{i_r}\circ\psi_{i_{r-1}}\circ\dotsb\circ\psi_{i_1}\circ\phi'(\x)\\
&=\e_{i_r}-\e_{i_{r-1}}+\e_{i_{r-2}}-\dotsb+(-1)^{r-1}\e_{i_1}+(-1)^r\phi'(\x)=\phi'(\x).
\end{align*}
In other words, $\phi=\phi'$.
\end{itemize}
\end{proof}

Miller, Ryan, and Ryj\'a\v{c}ek showed that if $G$ is a finite induced subgraph of $H_k$, then $\epsilon(G)\leq k$; also, if $G$ is a tree and $\epsilon(G)\leq k$, then $G$ is isomorphic to an induced subgraph of $H_k$. The following two theorems show the analogous statements for the sum index by replacing ``induced subgraph" with ``subgraph".

\begin{theorem}\label{thm:sumsubgraph}
Let $G$ be a finite subgraph of $H_k$. Then $s(G)\leq k$.
\end{theorem}

\begin{proof}
Let $r=\max\{\lVert\phi(\bz)\rVert:\phi\in V(G)\}$. Define $f:V(G)\to\Z$ such that for all $\phi\in V(G)$, $f(\phi)=x_1+(2r+1)x_2+\dotsb+(2r+1)^{k-1}x_k$, where $(x_1,x_2,\dotsc,x_k)=\phi(\bz)$. Since $|x_i|\leq r$ for all $1\leq i\leq k$, by viewing $f(\phi)$ as a base $2r+1$ expansion of an integer, it is easy to see that $f$ is injective. It remains to show that $|f^+|\leq k$.

Consider $\phi\phi'\in E(G)$, where $\phi=\psi_i\circ\phi'$ for some $1\leq i\leq k$. So if $\phi'(\bz)=(x_1',x_2',\dotsc,x_k')$, we have $\phi(\bz)=(-x_1',-x_2',\dotsc,1-x_i',\dotsc,-x_k')$, and
\begin{align*}
f^+(\phi\phi')&=f(\phi)+f(\phi')\\
&=\big(-x_1'+(2r+1)(-x_2')+\dotsb+(2r+1)^{i-1}(1-x_i')+\dotsb+(2r+1)^{k-1}(-x_k')\big)\\
&\hspace{20pt}+\big(x_1'+(2r+1)x_2'+\dotsb+(2r+1)^{k-1}x_k'\big)\\
&=(2r+1)^{i-1}.
\end{align*}
Therefore, $f^+(E(G))\subseteq\{1,2r+1,(2r+1)^2,\dotsc,(2r+1)^{k-1}\}$.
\end{proof}

\begin{theorem}\label{thm:treeembed}
If $G$ is a tree and $s(G)\leq k$, then $G$ is isomorphic to a subgraph of $H_k$.
\end{theorem}

\begin{proof}
Since $s(G)\leq k$, let $f$ be a vertex labeling of $G$ such that $f^+(E(G))\subseteq\{\alpha_1,\alpha_2,\dotsc,\alpha_k\}$. Let $\mathcal{I}:\{\alpha_1,\alpha_2,\dotsc,\alpha_k\}\to\{1,2,\dotsc,k\}$ be such that $\mathcal{I}(\alpha_i)=i$ for all $1\leq i\leq k$. Let $\mathcal{J}:E(G)\to\{1,2,\dotsc,k\}$ such that for every edge $uv\in E(G)$, $\mathcal{J}(uv)=\mathcal{I}(f^+(uv))$.

Fix a vertex $v_0\in V(G)$. By Corollary~\ref{cor:0vertex}, we may assume that $f(v_0)=0$. For any vertex $w\in V(G)$ of distance $r$ away from $v_0$, there exists a unique path $v_0v_1v_2\dotsb v_r$ from $v_0$ to $w$, where $v_r=w$. We define a map $\Phi:V(G)\to V(H_k)=\Gamma_k$ such that 
$$\Phi(w)=\psi_{\mathcal{J}(v_rv_{r-1})}\circ\psi_{\mathcal{J}(v_{r-1}v_{r-2})}\circ\dotsb\circ\psi_{\mathcal{J}(v_1v_0)}.$$

To verify that $\Phi$ is a graph homomorphism, consider two adjacent vertices $u$ and $w$ in $V(G)$. Then their distances away from $v_0$ differ by $1$. Without loss of generality, let the paths from $v_0$ to $u$ and from $v_0$ to $w$ be $v_0v_1v_2\dotsb v_rv_{r+1}$ and $v_0v_1v_2\dotsb v_r$, respectively, where $v_{r+1}=u$ and $v_r=w$. Hence, $\Phi(u)=\psi_{\mathcal{J}(v_{r+1}v_r)}\circ \Phi(w)$, or $\Phi(u)\circ\Phi(w)^{-1}=\psi_{\mathcal{J}(v_{r+1}v_r)}$. This shows that $\Phi(u)$ and $\Phi(w)$ are adjacent in $H_k$, since $\psi_{\mathcal{J}(v_{r+1}v_r)}$ is a generator of $\Gamma_k$.

It remains to verify that $\Phi$ is injective. We define a linear map $\mathcal{T}:\Z^k\to\Z$ such that
$$\mathcal{T}(x_1,x_2,\dotsc,x_k)=\alpha_1x_1+\alpha_2x_2+\dotsb\alpha_kx_k,$$
and for all $1\leq i\leq k$, define $\varphi_i:\Z\to\Z$ such that $\varphi_i(x)=\alpha_i-x$. Note that for every edge $uv\in E(G)$, \begin{equation}\label{eqn:varphi}
\varphi_{\mathcal{J}(uv)}(f(u))=\alpha_{\mathcal{J}(uv)}-f(u)=f^+(uv)-f(u)=f(v).
\end{equation}
Furthermore, for all $1\leq i\leq k$, $\mathcal{T}\circ\psi_i=\varphi_i\circ\mathcal{T}$, since
\begin{align*}
&\hspace{15.25pt}\mathcal{T}\circ\psi_i(x_1,x_2,\dotsc,x_k)\\
&=\mathcal{T}(\e_i-(x_1,x_2,\dotsc,x_k))\\
&=\alpha_1(-x_1)+\alpha_2(-x_2)+\dotsb+\alpha_{i-1}(-x_{i-1})+\alpha_i(1-x_i)+\alpha_{i+1}(-x_{i+1})+\dotsb+\alpha_k(-x_k)\\
&=\alpha_i-(\alpha_1x_1+\alpha_2x_2+\dotsb\alpha_kx_k)\\
&=\varphi_i\circ\mathcal{T}(x_1,x_2,\dotsc,x_k)
\end{align*}
for all $(x_1,x_2,\dotsc,x_k)\in\Z^k$.

Suppose $\Phi(u)=\Phi(w)$ for some vertices $u,w\in V(G)$. Let the paths from $v_0$ to $u$ and from $v_0$ to $w$ be $u_0u_1u_2\dotsb u_r$ and $w_0w_1w_2\dotsb w_s$, respectively, where $u_0=w_0=v_0$, $u_r=u$, and $w_s=w$. Then 
\begin{align*}
\mathcal{T}\circ\Phi(u)(\bz)&=\mathcal{T}\circ\psi_{\mathcal{J}(u_ru_{r-1})}\circ\psi_{\mathcal{J}(u_{r-1}u_{r-2})}\circ\dotsb\circ\psi_{\mathcal{J}(u_1u_0)}(\bz)\\
&=\varphi_{\mathcal{J}(u_ru_{r-1})}\circ\varphi_{\mathcal{J}(u_{r-1}u_{r-2})}\circ\dotsb\circ\varphi_{\mathcal{J}(u_1u_0)}\circ\mathcal{T}(\bz)\\
&=\varphi_{\mathcal{J}(u_ru_{r-1})}\circ\varphi_{\mathcal{J}(u_{r-1}u_{r-2})}\circ\dotsb\circ\varphi_{\mathcal{J}(u_1u_0)}(0)\\
&=\varphi_{\mathcal{J}(u_ru_{r-1})}\circ\varphi_{\mathcal{J}(u_{r-1}u_{r-2})}\circ\dotsb\circ\varphi_{\mathcal{J}(u_1u_0)}(f(u_0))\\
&=f(u_r),
\end{align*}
where the last equality is obtained by repeatedly applying equation~\eqref{eqn:varphi}. Similarly, $\mathcal{T}\circ\Phi(w)(\bz)=f(w_s)$. Since $\mathcal{T}\circ\Phi(u)(\bz)=\mathcal{T}\circ\Phi(w)(\bz)$, we have $f(u_r)=f(w_s)$. By the injectivity of $f$, we conclude that $u_r=w_s$, i.e., $u=w$. Therefore, $\Phi$ is injective, thus $G$ is isomorphic to a subgraph of $H_k$, as desired.
\end{proof}

In the remainder of this subsection, we establish a necessary condition, based on the density of vertices, for $G$ to be isomorphic to a subgraph of a hyperdiamond. This allows us to provide an improved lower bound of the sum index for trees, comparing to Theorem~\ref{thm:chromindex}. We start by establishing the following lemma.

\begin{lemma}\label{lem:characx}
Let $\phi'\in\Gamma_k$ be such that $\phi'(\bz)=\bz$. Let $\x=(x_1,x_2,\dotsc,x_k)\in\Z^k$, and let $r=\lVert\x\rVert$. Then there exists $\phi\in\Gamma_k$ such that $\phi(\bz)=\x$ if and only if $x_1+x_2+\dotsb+x_k=\frac{1+(-1)^{r+1}}{2}$.
\end{lemma}

\begin{proof}
If there exists $\phi\in\Gamma_k$ such that $\phi(\bz)=\x$, then by Lemma~\ref{lem:phiphi'}$(\ref{item:phi'=0})$, the distance between $\phi$ and $\phi'$ is $r=\lVert\x\rVert$. In other words,
$$\phi=\psi_{i_r}\circ\psi_{i_{r-1}}\circ\dotsb\circ\psi_{i_1}\circ\phi'$$
for some $i_1,i_2,\dotsc,i_r\in\{1,2,\dotsc,k\}$. Hence, $\phi(\bz)=\e_{i_r}-\e_{i_{r-1}}+\dotsb+(-1)^{r-1}\e_{i_1}$, and
$$x_1+x_2+\dotsb+x_k=\sum_{i=1}^r(-1)^{r-i}=\frac{1+(-1)^{r+1}}{2}.$$

If $x_1+x_2+\dotsb+x_k=\frac{1+(-1)^{r+1}}{2}$, then we can express
$$\x=\sum_{j\in S_1}\e_{i_j}-\sum_{j\in S_2}\e_{i_j},$$
where $S_1$ and $S_2$ forms a partition of $\{1,2,\dotsc,r\}$, $|S_1|-|S_2|=\frac{1+(-1)^{r+1}}{2}$, and $i_j\in\{1,2,\dotsc,k\}$ for all $j\in S_1\cup S_2$. In particular, we can choose $S_1=\{j\in\{1,2,\dotsc,r\}:r-j\text{ is even}\}$ and $S_2=\{j\in\{1,2,\dotsc,r\}:r-j\text{ is odd}\}$. We can now define $\phi=\psi_{i_r}\circ\psi_{i_{r-1}}\circ\dotsb\circ\psi_{i_1}\circ\phi'$.
As a result, $\phi\in\Gamma_k$ and
\begin{align*}
\phi(\bz)&=\psi_{i_r}\circ\psi_{i_{r-1}}\circ\dotsb\circ\psi_{i_1}\circ\phi'(\bz)\\
&=\e_{i_r}-\e_{i_{r-1}}+\dotsb+(-1)^{r-1}\e_{i_1}\\
&=\sum_{j\in S_1}\e_{i_j}-\sum_{j\in S_2}\e_{i_j}=\x.
\end{align*}
\end{proof}

\begin{theorem}\label{thm:density}
Let $r$ be a positive integer. Then the number of vertices in $H_k$ that are of distance $r$ away from a fixed vertex $\phi\in V(H_k)$ is
$$\sum_{j=1}^k\binom{k}{j}\binom{\lceil r/2\rceil+j-1}{j-1}\binom{\lfloor r/2\rfloor-1}{k-j-1}.$$
Here, we define $\binom{-1}{x}=0$ for all nonnegative integers $x$ and $\binom{-1}{-1}=1$.
\end{theorem}

\begin{proof}
Since $H_k$ is vertex-transitive, we will only consider the case where the fixed vertex is $\phi'$ such that $\phi'(\bz)=\bz$. Let $\mathcal{V}$ be the set of vertices in $H_k$ that are of distance $r$ away from $\phi'$, and let 
$$\mathcal{V}'=\left\{\x=(x_1,x_2,\dotsc,x_k)\in\Z^k:\lVert\x\rVert=r\text{ and }x_1+x_2+\dotsb+x_k=\frac{1+(-1)^{r+1}}{2}\right\}.$$
Let $\tau:\mathcal{V}\to\mathcal{V}'$ be such that $\tau(\phi)=\phi(\bz)$ for all $\phi\in\mathcal{V}$. We will show that $\tau$ is a bijection. For all $\phi\in\mathcal{V}$,  $\lVert\phi(\bz)\rVert=r$ by Lemma~\ref{lem:phiphi'}$(\ref{item:phi'=0})$, and $x_1+x_2+\dotsb+x_k=\frac{1+(-1)^{r+1}}{2}$ by Lemma~\ref{lem:characx}, so $\tau$ is well-defined. If $\phi,{\phi}^*\in\mathcal{V}$ such that $\tau(\phi)=\tau(\phi^*)$, then $\phi(\bz)=\phi^*(\bz)$, which implies that $\phi=\phi^*$ by Lemma~\ref{lem:phiphi'}$(\ref{item:phi=phi'})$, so $\tau$ is injective. For all $\x=(x_1,x_2,\dotsc,x_k)\in\mathcal{V}'$, there exists $\phi\in\Gamma_k$ such that $\phi(\bz)=\x$ by Lemma~\ref{lem:characx}, and $\phi$ is of distance $r$ away from $\phi'$ by Lemma~\ref{lem:phiphi'}$(\ref{item:phi'=0})$. Therefore, $\phi\in\mathcal{V}$ and $\tau(\phi)=\x$, so $\tau$ is surjective.

Since $\tau:\mathcal{V}\to\mathcal{V}'$ is a bijection, we have $|\mathcal{V}|=|\mathcal{V}'|$. To count the number of elements in $\mathcal{V}'$, we first partition $\mathcal{V}'$ into $\mathcal{V}'_0,\mathcal{V}'_1,\mathcal{V}'_2,\dotsc,\mathcal{V}'_k$, where
$$\mathcal{V}'_j=\{\x=(x_1,x_2,\dotsc,x_k)\in\mathcal{V}':\text{the number of nonnegative entries in \x\ is }j\}.$$
Note that $\mathcal{V}'_0$ is empty since $x_1+x_2+\dotsb+x_k=\frac{1+(-1)^{r+1}}{2}\geq0$. For each fixed $j\in\{1,2,\dotsc,k\}$, there are $\binom{k}{j}$ ways to partition $\{1,2,\dotsc,k\}$ into two subsets $S_1$ and $S_2$ such that $|S_1|=j$ and $|S_2|=k-j$. For each such partition $S_1$ and $S_2$, let
$$\mathcal{V}'_{S_1}=\{\x=(x_1,x_2,\dotsc,x_k)\in\mathcal{V}'_j:x_i\geq0\text{ if and only if }i\in S_1\}.$$
It is easy to see that the number of elements in $\mathcal{V}'_{S_1}$ is the same as the number of integer solutions $(x_1,x_2,\dotsc,x_k)$ to the system
$$\begin{cases}
\displaystyle\sum_{i\in S_1}x_i-\sum_{i\in S_2}x_i=r,\\
\displaystyle\sum_{i\in S_1}x_i+\sum_{i\in S_2}x_i=\frac{1+(-1)^{r+1}}{2}
\end{cases}$$
such that $x_i\geq0$ if and only if $i\in S_1$. 
This system is equivalent to
\begin{numcases}
\displaystyle\sum_{i\in S_1}x_i=\frac{r}{2}+\frac{1+(-1)^{r+1}}{4}=\left\lceil\frac{r}{2}\right\rceil,\label{eqn:S1}\\
-\displaystyle\sum_{i\in S_2}x_i=\frac{r}{2}-\frac{1+(-1)^{r+1}}{4}=\left\lfloor\frac{r}{2}\right\rfloor.\label{eqn:S2}
\end{numcases}
The number of nonnegative integer solutions to equation~\eqref{eqn:S1} is $\binom{\lceil r/2\rceil+j-1}{j-1}$, and the number of negative integer solutions to equation~\eqref{eqn:S2} is $\binom{\lfloor r/2\rfloor-1}{k-j-1}$. Therefore, the cardinality of $\mathcal{V}'_j$ is $\binom{k}{j}\binom{\lceil r/2\rceil+j-1}{j-1}\binom{\lfloor r/2\rfloor-1}{k-j-1}$, thus proving the theorem by summing over $j\in\{1,2,\dotsc,k\}$.
\end{proof}

Combining Theorems~\ref{thm:treeembed} and \ref{thm:density}, we have the following corollary.

\begin{corollary}\label{cor:tree}
Let $G$ be a tree. Then $s(G)\geq k$, where $k$ is the minimum positive integer such that for every vertex $v\in V(G)$ and positive integer $r$, the number of vertices in $G$ that are of distance $r$ away from $v$ is at most
$$\sum_{j=1}^k\binom{k}{j}\binom{\lceil r/2\rceil+j-1}{j-1}\binom{\lfloor r/2\rfloor-1}{k-j-1}.$$
\end{corollary}

The last corollary of this subsection illustrates an application of Corollary~\ref{cor:tree}.

\begin{corollary}
For every positive integer $k$, there exists a binary tree $G$ such that $s(G)>k$.
\end{corollary}

\begin{proof}
Consider a fixed positive integer $k$. Let $G_r$ be a perfect binary tree with height $r$. The number of vertices in $G_r$ that are of distance $r$ away from the root vertex is $2^r$. However, if $s(G)\leq k$, then the number of vertices allowed by Corollary~\ref{cor:tree} is a polynomial in $r$. This leads to a contradiction when $r$ is sufficiently large.
\end{proof}

\section{Difference index}\label{sec:diffindex}

The main goal of this section is to develop results regarding the difference index of graphs analogous to those in Section~\ref{sec:sumindex}.

\subsection{Bounds on the difference index}

We begin by presenting a lower bound for the difference index of a graph $G$.

\begin{theorem}\label{thm:dilowerbnd}
Let $\delta(G)$ be the minimum degree of $G$, and recall that $\chi'(G)$ is the chromatic index of $G$. Then $d(G)\geq\max\left\{\left\lceil\frac{\chi'(G)}{2}\right\rceil,\delta(G)\right\}$.
\end{theorem}

\begin{proof}
Let $f$ be a difference index labeling of $G$, and let $v_0$ be the vertex of $G$ at which $f$ attains its maximum. Let $v_1,v_2,\dotsc,v_r$ be the neighbors of $v_0$, where $r\geq\delta(G)$. Then $f^-(v_0v_i)=f(v_0)-f(v_i)$ must be distinct for all $1\leq i\leq r$ since $f$ is injective. Therefore, $d(G)\geq r\geq\delta(G)$.

Consider $\alpha\in f^-(E(G))$. Let $G_\alpha$ be the subgraph of $G$ such that 
$$E(G_\alpha)=\{e\in E(G):f^-(e)=\alpha\}.$$
Note that the maximum degree of $G_\alpha$ is at most $2$; otherwise, if $u_0\in V(G_\alpha)$ has neighbors $u_1$, $u_2$, and $u_3$ in $G_\alpha$, then $f^-(u_0u_1)=f^-(u_0u_2)=f^-(u_0u_3)=\alpha$ will contradict that $f$ is injective. Moreover, since $1=d(G_\alpha)\geq\delta(G_\alpha)$, $G_\alpha$ can only be a disjoint union of paths. 

As a result, we can define a proper edge coloring $c_\alpha:E(G_\alpha)\to\Z$ such that the image of $c_\alpha$ is $\{\alpha,-\alpha\}$. Finally, define $c:E(G)\to\Z$ such that for each $\alpha\in f^-(E(G))$, $c(e)=c_\alpha(e)$ if $e\in G_\alpha$. It is clear that $c$ is a proper edge coloring of $G$. Hence, $\chi'(G)\leq2d(G)$.
\end{proof}

We next provide an analogue of Lemma~\ref{lem:shifting}, which we will find useful in this section.  We state the lemma without proof, as the proof is similar to that of Lemma~\ref{lem:shifting}.

\begin{lemma}\label{lem:dishifting}
Let $f$ be a vertex labeling of $G$. Let $g$ be a vertex labeling of $G$ such that for all vertices $v\in V(G)$, $g(v)=f(v)+c$ for some integer $c$. Then $|f^-|=|g^-|$.
\end{lemma}

If $G$ is a bipartite graph, we have a bound on the difference index in terms of its sum index.

\begin{theorem}\label{thm:dibipartite}
Let $G$ be a bipartite graph. Then $\left\lceil\frac{s(G)}{2}\right\rceil\leq d(G)\leq s(G)$.
\end{theorem}

\begin{proof}
Let $A$ and $B$ be the two partite sets of $G$. Let $f$ be a vertex labeling of $G$, and in view of Lemmas~\ref{lem:shifting} and \ref{lem:dishifting}, we may assume that $f(v)>0$ for all $v\in V(G)$.

Consider $g:V(G)\to\Z$ such that 
$$g(v)=\begin{cases}
f(v)&\text{if }v\in A,\\
-f(v)&\text{if }v\in B.
\end{cases}$$
Note that $g$ is injective, since $f$ is injective and $f(v)>0$ for all $v\in V(G)$. For any edge $uv\in E(G)$, assume that $u\in A$ and $v\in B$. Then
$$g^-(uv)=|g(u)-g(v)|=|f(u)-(-f(v))|=f(u)+f(v)=f^+(uv)$$
and
$$g^+(uv)=g(u)+g(v)=f(u)-f(v)=\pm f^-(uv).$$
This implies that 
$$g^-(E(G))=f^+(E(G))$$
and
$$g^+(E(G))\subseteq\{x,-x\in\Z:x\in f^-(E(G))\}.$$
Therefore, $d(G)\leq s(G)$ and $s(G)\leq2d(G)$, and the result follows.
\end{proof}

\subsection{Difference index of complete graphs, complete bipartite graphs, caterpillars, cycles, spiders, wheels, and rectangular grids}

Parallel to Section~\ref{sec:sumindex}, we determine the difference index for several families of graphs. Similar to the sum index, an exact difference index can be found for certain families of graphs.

The following corollary is an immediate consequence of Lemma~\ref{lem:dishifting}.

\begin{corollary}\label{cor:di0vertex}
Let $v$ be a vertex of $G$. There exists a difference index labeling $f$ such that $f(v)=0$.
\end{corollary}

In the rest of this subsection, we determine the difference index of complete graphs, complete bipartite graphs, caterpillars, cycles, spiders, wheels, and rectangular grids.

\begin{theorem}\label{thm:dicomplete}
For every complete graph $K_n$, $d(K_n)=n-1$.
\end{theorem}

\begin{proof}
Let $v_1,v_2,\dotsc,v_n$ be the vertices of $K_n$. Define $f$ to be a vertex labeling of $K_n$ such that $f(v_i)=i$ for all $1\leq i\leq n$. Then $f^-(v_iv_j)=|i-j|\in\{1,2,\dotsc,n-1\}$, so $d(K_n)\leq|f^-|=n-1$. The result follows since $d(K_n)\geq\delta(K_n)=n-1$ by Theorem~\ref{thm:dilowerbnd}.
\end{proof}

\begin{theorem}
For every complete bipartite graph $K_{n,m}$, $d(K_{n,m})=\left\lceil\frac{n+m-1}{2}\right\rceil$.
\end{theorem}

\begin{proof}
As mentioned in the introduction, Harrington and Wong showed that $s(K_{n,m})=n+m-1$. By Theorem~\ref{thm:dibipartite}, we have $d(K_{n,m})\geq\left\lceil\frac{n+m-1}{2}\right\rceil$. Hence, it suffices to find a vertex labeling $f$ such that $|f^-|\leq\left\lceil\frac{n+m-1}{2}\right\rceil$.

Let $u_1,u_2,\dotsc,u_n,v_1,v_2,\dotsc,v_m$ be the vertices of $K_{n,m}$, and assume without loss of generality that $n$ is even or $m$ is odd. Define $f:V(K_{n,m})\to\Z$ such that $f(u_i)=-2\left\lceil\frac{n}{2}\right\rceil+2i-1$ and $f(v_j)=-2\left\lceil\frac{m}{2}\right\rceil+2j$ for all $1\leq i\leq n$ and $1\leq j\leq m$. Note that $f$ is injective since $f(u_i)$ is odd and $f(v_j)$ is even for all $1\leq i\leq n$ and $1\leq j\leq m$. Furthermore, for all $1\leq i\leq n$ and $1\leq j\leq m$,
\begin{equation}\label{eqn:dibipartite}
f(u_i)-f(v_j)=2\left\lceil\frac{m}{2}\right\rceil-2\left\lceil\frac{n}{2}\right\rceil-1+2(i-j),
\end{equation}
which is an odd integer. Equation~\eqref{eqn:dibipartite} attains its maximum when $i=n$ and $j=1$ and its minimum when $i=1$ and $j=m$. Hence,
\begin{align*}
&\hspace{14pt}\max\{f^-(u_iv_j):1\leq i\leq n,1\leq j\leq m\}\\
&=\max\left\{\left|2\left\lceil\frac{m}{2}\right\rceil-2\left\lceil\frac{n}{2}\right\rceil-1+2(n-1)\right|,\left|2\left\lceil\frac{m}{2}\right\rceil-2\left\lceil\frac{n}{2}\right\rceil-1+2(1-m)\right|\right\}\\
&=\max\left\{\left|2\left\lceil\frac{m}{2}\right\rceil+2\left\lfloor\frac{n}{2}\right\rfloor-3\right|,\left|2\left\lfloor\frac{m}{2}\right\rfloor+2\left\lceil\frac{n}{2}\right\rceil-1\right|\right\}\\
&=\max\left\{2\left\lceil\frac{m}{2}\right\rceil+2\left\lfloor\frac{n}{2}\right\rfloor-3,2\left\lfloor\frac{m}{2}\right\rfloor+2\left\lceil\frac{n}{2}\right\rceil-1\right\}\\
&=2\left\lfloor\frac{m}{2}\right\rfloor+2\left\lceil\frac{n}{2}\right\rceil-1,
\end{align*}
where the second last equality is due to $2\left\lceil\frac{m}{2}\right\rceil+2\left\lfloor\frac{n}{2}\right\rfloor-3\geq-1$ and $2\left\lfloor\frac{m}{2}\right\rfloor+2\left\lceil\frac{n}{2}\right\rceil-1\geq1$, and the last equality is due to $2\left\lceil\frac{m}{2}\right\rceil-3\leq2\left\lfloor\frac{m}{2}\right\rfloor-1$ and $2\left\lfloor\frac{n}{2}\right\rfloor\leq2\left\lceil\frac{n}{2}\right\rceil$. As a result, 
$$f^-(E(K_{n,m}))\subseteq\left\{x\in\N:x\text{ is odd and }x\leq2\left\lfloor\frac{m}{2}\right\rfloor+2\left\lceil\frac{n}{2}\right\rceil-1\right\},$$
and hence,
\begin{align*}
|f^-|&\leq\left\lfloor\frac{m}{2}\right\rfloor+\left\lceil\frac{n}{2}\right\rceil\\
&=\begin{cases}
\left\lceil\frac{m-1}{2}\right\rceil+\frac{n}{2}&\text{if $n$ is even},\\
\frac{m-1}{2}+\left\lceil\frac{n}{2}\right\rceil&\text{if $n$ is odd}
\end{cases}\\
&=\left\lceil\frac{n+m-1}{2}\right\rceil.
\end{align*}
\end{proof}

A caterpillar graph is a tree that consists of a central path and every other vertex is distance $1$ away from a cut vertex on this central path.

\begin{theorem}
The difference index of a caterpillar graph $G$ is $\left\lceil\frac{\Delta}{2}\right\rceil$, where $\Delta$ is the maximum degree of $G$.
\end{theorem}

\begin{proof}
By Theorem~\ref{thm:dilowerbnd}, $d(G)\geq\left\lceil\frac{\chi'(G)}{2}\right\rceil\geq\left\lceil\frac{\Delta}{2}\right\rceil$. Hence, it suffices to find a vertex labeling $f:V(G)\to\Z$ such that $|f^-|=\left\lceil\frac{\Delta}{2}\right\rceil$.

Let $v_1,v_2,\dotsc,v_n$ be the vertices of the central path of $G$. For each $2\leq i\leq n-1$, if $\deg(v_i)\geq3$, then let $\{u_{ij}:1\leq j\leq\deg(v_i)-2\}$ be the set of neighbors of $v_i$ other than $v_{i-1}$ and $v_{i+1}$. Let $f:V(G)\to\N$ be defined such that $f(v_i)=i\Delta$ for all $1\leq i\leq n$ and
$$f(u_{ij})=\begin{cases}
i\Delta+j-\left\lfloor\frac{\deg(v_i)}{2}\right\rfloor&\text{if }1\leq j\leq\left\lfloor\frac{\deg(v_i)-2}{2}\right\rfloor,\vspace{1pt}\\
i\Delta+j-\left\lfloor\frac{\deg(v_i)-2}{2}\right\rfloor&\text{if }\left\lfloor\frac{\deg(v_i)}{2}\right\rfloor\leq j\leq\deg(v_i)-2.\\
\end{cases}$$
Figure~\ref{fig:caterpillar} shows the vertex labeling $f$ on a caterpillar graph $G$.
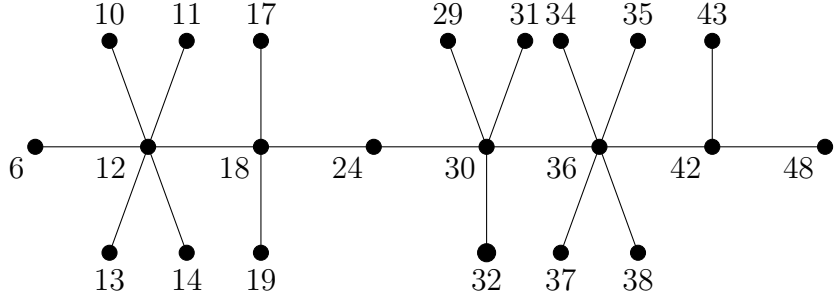
\begin{figure}[H]
\centering
\begin{tikzpicture}[scale=1.5]
\def\r{0.1/1.5};
\foreach \i in {0,1,2,...,7}{\filldraw (\i,0) circle (\r);}; \draw (0,0)--(7,0); \filldraw ({1+cos(70)},{sin(70)}) circle (\r); \filldraw ({1+cos(110)},{sin(110)}) circle (\r); \filldraw ({1+cos(250)},{sin(250)}) circle (\r); \filldraw ({1+cos(290)},{sin(290)}) circle (\r); \draw ({1+cos(70)},{sin(70)})--({1+cos(250)},{sin(250)}); \draw ({1+cos(110)},{sin(110)})--({1+cos(290)},{sin(290)}); \filldraw ({2+cos(90)},{sin(70)}) circle (\r); \filldraw ({2+cos(270)},{sin(250)}) circle (\r); \draw ({2+cos(90)},{sin(70)})--({2+cos(270)},{sin(250)}); \filldraw ({4+cos(70)},{sin(70)}) circle (\r); \filldraw ({4+cos(110)},{sin(110)}) circle (\r); \filldraw ({4+cos(270)},{sin(250)}) circle (0.08); \draw ({4+cos(70)},{sin(70)})--(4,0); \draw ({4+cos(110)},{sin(110)})--(4,0); \draw ({4+cos(270)},{sin(250)})--(4,0); \filldraw ({5+cos(70)},{sin(70)}) circle (\r); \filldraw ({5+cos(110)},{sin(70)}) circle (\r); \filldraw ({5+cos(250)},{sin(250)}) circle (\r); \filldraw ({5+cos(290)},{sin(290)}) circle (\r); \draw ({5+cos(70)},{sin(70)})--({5+cos(250)},{sin(250)}); \draw ({5+cos(110)},{sin(110)})--({5+cos(290)},{sin(290)}); \filldraw ({6+cos(90)},{sin(70)}) circle (\r); \draw ({6+cos(90)},{sin(70)})--(6,0); \node[below left] at (0,0){$6$}; \node[below left] at (0.9,0){$12$}; \node[below left] at (2,0){$18$}; \node[below left] at (3,0){$24$}; \node[below left] at (4,0){$30$}; \node[below left] at (4.9,0){$36$}; \node[below left] at (6,0){$42$}; \node[below left] at (7,0){$48$}; \node[above] at ({1+cos(110)},{sin(110)+0.05}){$10$}; \node[above] at ({1+cos(70)},{sin(70)+0.05}){$11$}; \node[below] at ({1+cos(250)},{sin(250)-0.05}){$13$}; \node[below] at ({1+cos(290)},{sin(290)-0.05}){$14$}; \node[above] at ({2+cos(90)},{sin(70)+0.05}){$17$}; \node[below] at ({2+cos(270)},{sin(250)-0.05}){$19$}; \node[above] at ({4+cos(110)},{sin(110)+0.05}){$29$}; \node[above] at ({4+cos(70)},{sin(70)+0.05}){$31$}; \node[below] at ({4+cos(270)},{sin(250)-0.05}){$32$}; \node[above] at ({5+cos(110)},{sin(110)+0.05}){$34$}; \node[above] at ({5+cos(70)},{sin(70)+0.05}){$35$}; \node[below] at ({5+cos(250)},{sin(250)-0.05}){$37$}; \node[below] at ({5+cos(290)},{sin(290)-0.05}){$38$}; \node[above] at ({6+cos(90)},{sin(70)+0.05}){$43$};
\end{tikzpicture}
\caption{A caterpillar graph with $\Delta=6$ and vertex labeling $f$}
\label{fig:caterpillar}
\end{figure}

To show that $f$ is injective, we note that for each $2\leq i\leq n-1$, $f(u_{ij})<i\Delta=f(v_i)$ if $1\leq j\leq\left\lfloor\frac{\deg(v_i)-2}{2}\right\rfloor$ and $f(u_{ij})>i\Delta=f(v_i)$ if $\left\lfloor\frac{\deg(v_i)}{2}\right\rfloor\leq j\leq\deg(v_i)-2$. Moreover, for each $2\leq i\leq n-2$, $1\leq j\leq\deg(v_i)-2$, and $1\leq j'\leq\deg(v_{i+1})-2$,
\begin{align*}
f(u_{ij})\leq i\Delta+\left\lceil\tfrac{\deg(v_i)-2}{2}\right\rceil&\leq i\Delta+\left\lceil\tfrac{\Delta}{2}\right\rceil-1\\
&<(i+1)\Delta+1-\left\lfloor\tfrac{\Delta}{2}\right\rfloor\leq(i+1)\Delta+1-\left\lfloor\tfrac{\deg(v_{i+1})}{2}\right\rfloor\leq f(u_{(i+1)j'}).
\end{align*}
Hence, $f$ is injective. As for the range of $f^-$, for each $2\leq i\leq n-1$, $f^-(v_iu_{ij})=f(v_i)-f(u_{ij})\in\left\{1,2,\dotsc,\left\lfloor\frac{\deg(v_i)}{2}\right\rfloor-1\right\}$ if $1\leq j\leq\left\lfloor\frac{\deg(v_i)-2}{2}\right\rfloor$, and $f^-(v_iu_{ij})=f(u_{ij})-f(v_i)\in\left\{1,2,\dotsc,\left\lceil\frac{\deg(v_i)-2}{2}\right\rceil\right\}$ if $\left\lfloor\frac{\deg(v_i)}{2}\right\rfloor\leq j\leq\deg(v_i)-2$. Furthermore, for each $1\leq i\leq n-1$, $f(v_iv_{i+1})=f(v_{i+1})-f(v_i)=\Delta$. Therefore, the range of $f^-$ is $\left\{1,2,\dotsc,\left\lceil\frac{\Delta-2}{2}\right\rceil,\Delta\right\}$, which has cardinality $\left\lceil\frac{\Delta}{2}\right\rceil$.
\end{proof}

\begin{theorem}\label{thm:dicycle}
Let $n\geq3$ be an integer. Then $d(C_n)=2$.
\end{theorem}

\begin{proof}
Let $C_n$ be the cycle $v_0v_1v_2\dotsb v_{n-1}v_0$. Define $f$ to be a vertex labeling of $C_n$ such that $f(v_i)=i$ for all $0\leq i\leq n-1$. Then $f^-(v_iv_{i+1})=|f(v_i)-f(v_{i+1})|=1$ for all $0\leq i\leq n-2$, and $f^-(v_{n-1}v_0)=|f(v_{n-1})-f(v_0)|=n-1\neq1$. As a result, $|f^-|=2$, so $d(C_n)\leq2$. The result follows since $d(C_n)\geq\delta(C_n)=2$ by Theorem~\ref{thm:dilowerbnd}.
\end{proof}

With a similar vertex labeling as in the proof of Theorem~\ref{thm:dicycle}, and together with the proof of Theorem~\ref{thm:dilowerbnd}, we have the following corollary.

\begin{corollary}\label{cor:dipath}
A graph $G$ satisfies $d(G)=1$ if and only if $G$ is a disjoint union of paths.
\end{corollary}

\begin{theorem} 
For every spider $S_{\ell_1,\ell_2,\dotsc,\ell_\Delta}$, $d(S_{\ell_1,\ell_2,\dotsc,\ell_\Delta})=\left\lceil\frac{\Delta}{2}\right\rceil$.
\end{theorem}

\begin{proof}
By Theorem~\ref{thm:dilowerbnd}, $d(S_{\ell_1,\ell_2,\dotsc,\ell_\Delta})\geq\left\lceil\frac{\chi'(S_{\ell_1,\ell_2,\dotsc,\ell_\Delta})}{2}\right\rceil\geq\left\lceil\frac{\Delta}{2}\right\rceil$. Hence, it suffices to find a vertex labeling $f$ such that $|f^-|=\left\lceil\frac{\Delta}{2}\right\rceil$.

Similar to the proof of Theorem~\ref{thm:spider}, let $\xi$ be the smallest nonnegative integer such that $\Delta\mod{\xi}{2}$, and let $\alpha=\frac{\Delta+\xi}{2}$. Define $f(v_0)=0$. For all $1\leq i\leq\Delta$ and $1\leq j\leq\ell_i$, define
$$f(v_{i,j})=(-1)^{\Delta-i}\left((j-1)\alpha+\left\lceil\frac{i+\xi}{2}\right\rceil\right).$$
Figure~\ref{fig:dispider} shows the vertex labeling $f$ on the spider $S_{3,1,2,4,2,3,4}$.
\begin{figure}[H]
\centering
\begin{tikzpicture}[scale=1.5]
\foreach\r in{0.1/1.5}{\filldraw(0,0)circle(\r);\foreach\i in{1,2,3}{\filldraw({\i*cos(0*360/7)},{\i*sin(0*360/7)})circle(\r);}\foreach\i in{1}{\filldraw({\i*cos(1*360/7)},{\i*sin(1*360/7)})circle(\r);}\foreach\i in{1,2}{\filldraw({\i*cos(2*360/7)},{\i*sin(2*360/7)})circle(\r);}\foreach\i in{1,2,3,4}{\filldraw({\i*cos(3*360/7)},{\i*sin(3*360/7)})circle(\r);}\foreach\i in{1,2}{\filldraw({\i*cos(4*360/7)},{\i*sin(4*360/7)})circle(\r);}\foreach\i in{1,2,3}{\filldraw({\i*cos(5*360/7)},{\i*sin(5*360/7)})circle(\r);}\foreach\i in{1,2,3,4}{\filldraw({\i*cos(6*360/7)},{\i*sin(6*360/7)})circle(\r);}}\foreach[count=\i]\j in{3,1,2,4,2,3,4}{\draw(0,0)--({\j*cos((\i-1)*360/7)},{\j*sin((\i-1)*360/7)});}
\node[left]at(-0.1,0){$0$};\node[above right]at({1*cos(0*360/7)},{1*sin(0*360/7)}){$1$};\node[above right]at({2*cos(0*360/7)},{2*sin(0*360/7)}){$5$};\node[above right]at({3*cos(0*360/7)},{3*sin(0*360/7)}){$9$};\node[above right]at({1*cos(1*360/7)},{1*sin(1*360/7)}){$-2$};\node[above right]at({1*cos(2*360/7)},{1*sin(2*360/7)}){$2$};\node[above right]at({2*cos(2*360/7)},{2*sin(2*360/7)}){$6$};\node[above right]at({1*cos(3*360/7)},{1*sin(3*360/7)}){$-3$};\node[above right]at({2*cos(3*360/7)},{2*sin(3*360/7)}){$-7$};\node[above right]at({3*cos(3*360/7)},{3*sin(3*360/7)}){$-11$};\node[above right]at({4*cos(3*360/7)},{4*sin(3*360/7)}){$-15$};\node[above left]at({1*cos(4*360/7)},{1*sin(4*360/7)}){$3$};\node[above left]at({2*cos(4*360/7)},{2*sin(4*360/7)}){$7$};\node[above left]at({1*cos(5*360/7)},{1*sin(5*360/7)}){$-4$};\node[above left]at({2*cos(5*360/7)},{2*sin(5*360/7)}){$-8$};\node[above left]at({3*cos(5*360/7)},{3*sin(5*360/7)}){$-12$};\node[above right]at({1*cos(6*360/7)},{1*sin(6*360/7)}){$4$};\node[above right]at({2*cos(6*360/7)},{2*sin(6*360/7)}){$8$};\node[above right]at({3*cos(6*360/7)},{3*sin(6*360/7)}){$12$};\node[above right]at({4*cos(6*360/7)},{4*sin(6*360/7)}){$16$};
\end{tikzpicture}
\caption{The vertex labeling $f$ on the spider $S_{3,1,2,4,2,3,4}$}
\label{fig:dispider}
\end{figure}

The proof that $f$ is injective and hence a vertex labeling is the same as in the proof of Theorem~\ref{thm:spider}. To verify that $|f^-|=\left\lceil\frac{\Delta}{2}\right\rceil$, we show that
$$f^-(E(S_{\ell_1,\ell_2,\dotsc,\ell_\Delta}))=\{f^-(v_0v_{i,1}):1\leq i\leq\Delta\}=\left\{1,2,\dotsc,\left\lceil\frac{\Delta}{2}\right\rceil\right\},$$
where the second equality is obvious. For each $1\leq i\leq\Delta$ and $1\leq j\leq\ell_i-1$,
$$f^-(v_{i,j}v_{i,j+1})=\left|(-1)^{\Delta-i}\left((j-1)\alpha+\left\lceil\frac{i+\xi}{2}\right\rceil\right)-(-1)^{\Delta-i}\left(j\alpha+\left\lceil\frac{i+\xi}{2}\right\rceil\right)\right|=\alpha,$$
which is an element of $\{f^-(v_0v_{i,1}):1\leq i\leq\Delta\}$ since $\alpha=\left\lceil\frac{\Delta}{2}\right\rceil$.
\end{proof}

\begin{theorem}\label{thm:diwheel}
Let $\Delta\geq3$ be an integer, and let $W_{\Delta}$ be the wheel graph with maximum degree $\Delta$. Then $d(W_\Delta)=\max\left\{3,\left\lceil\frac{\Delta}{2}\right\rceil\right\}$.
\end{theorem}

\begin{proof}
Since $W_3$ is isomorphic to the complete graph $K_4$, by Theorem~\ref{thm:dicomplete}, we have $d(W_3)=d(K_4)=4-1=3$. Next, we will show that $d(W_4)\geq3$. First, $d(W_4)\geq\left\lceil\frac{\chi'(W_4)}{2}\right\rceil\geq\left\lceil\frac{4}{2}\right\rceil=2$ by Theorem~\ref{thm:dilowerbnd}. Assume for the sake of contradiction that $f$ is a difference index labeling of $W_4$, where the image of $f^-$ is the set $\{\alpha,\beta\}$ and $0<\alpha<\beta$. By Corollary~\ref{cor:di0vertex}, we may further assume that $f(v_0)=0$. Without loss of generality, the only three candidates for $f^-$ are given by Figures~\ref{fig:diwheel1}, \ref{fig:diwheel2}, and \ref{fig:diwheel3}.
\begin{figure}[H]
\centering
\begin{minipage}{0.45\linewidth}
\centering
\begin{tikzpicture}
\coordinate(v0)at(0,0);\coordinate(v1)at(2,0);\coordinate(v2)at(0,2);\coordinate(v3)at(-2,0);\coordinate(v4)at(0,-2);\foreach\i in{v0,v1,v2,v3,v4}{\filldraw(\i)circle(0.1);}\draw(v3)--(v1)--(v2)--(v3)--(v4)--(v1);\draw(v2)--(v4);\node[below right]at(v0){$0$};\node[above right]at(v1){$\alpha$};\node[above right]at(v2){$\beta$};\node[above left]at(v3){$-\alpha$};\node[below right]at(v4){$-\beta$};\node[above]at(1,0){$\alpha$};\node[left]at(0,1){$\beta$};\node[above]at(-1,0){$\alpha$};\node[left]at(0,-1){$\beta$};\node[above right]at(1,1){$\beta-\alpha$};\node[above left]at(-1,1){$\alpha+\beta$};\node[below left]at(-1,-1){$\beta-\alpha$};\node[below right]at(1,-1){$\alpha+\beta$};
\end{tikzpicture}
\caption{Candidate $1$ for $f^-$ on $W_4$}
\label{fig:diwheel1}
\end{minipage}\qquad
\begin{minipage}{0.45\linewidth}
\centering
\begin{tikzpicture}
\coordinate(v0)at(0,0);\coordinate(v1)at(2,0);\coordinate(v2)at(0,2);\coordinate(v3)at(-2,0);\coordinate(v4)at(0,-2);\foreach\i in{v0,v1,v2,v3,v4}{\filldraw(\i)circle(0.1);}\draw(v3)--(v1)--(v2)--(v3)--(v4)--(v1);\draw(v2)--(v4);\node[below right]at(v0){$0$};\node[above right]at(v1){$\alpha$};\node[above right]at(v2){$-\alpha$};\node[above left]at(v3){$\beta$};\node[below right]at(v4){$-\beta$};\node[above]at(1,0){$\alpha$};\node[left]at(0,1){$\alpha$};\node[above]at(-1,0){$\beta$};\node[left]at(0,-1){$\beta$};\node[above right]at(1,1){$2\alpha$};\node[above left]at(-1,1){$\alpha+\beta$};\node[below left]at(-1,-1){$2\beta$};\node[below right]at(1,-1){$\alpha+\beta$};
\end{tikzpicture}
\caption{Candidate $2$ for $f^-$ on $W_4$}
\label{fig:diwheel2}
\end{minipage}\qquad
\begin{minipage}{0.45\linewidth}
\centering
\begin{tikzpicture}
\coordinate(v0)at(0,0);\coordinate(v1)at(2,0);\coordinate(v2)at(0,2);\coordinate(v3)at(-2,0);\coordinate(v4)at(0,-2);\foreach\i in{v0,v1,v2,v3,v4}{\filldraw(\i)circle(0.1);}\draw(v3)--(v1)--(v2)--(v3)--(v4)--(v1);\draw(v2)--(v4);\node[below right]at(v0){$0$};\node[above right]at(v1){$\alpha$};\node[above right]at(v2){$-\alpha$};\node[above left]at(v3){$-\beta$};\node[below right]at(v4){$\beta$};\node[above]at(1,0){$\alpha$};\node[left]at(0,1){$\alpha$};\node[above]at(-1,0){$\beta$};\node[left]at(0,-1){$\beta$};\node[above right]at(1,1){$2\alpha$};\node[above left]at(-1,1){$\beta-\alpha$};\node[below left]at(-1,-1){$2\beta$};\node[below right]at(1,-1){$\beta-\alpha$};
\end{tikzpicture}
\caption{Candidate $3$ for $f^-$ on $W_4$}
\label{fig:diwheel3}
\end{minipage}
\end{figure}
\noindent Since the sets $\{\alpha,\beta,\alpha+\beta\}$ and $\{\alpha,\beta,2\beta\}$ are of cardinality $3$, we see that $|f^-|\geq3$ in each of these figures, which leads to a contradiction. Hence, $d(W_4)\geq3$. Furthermore, letting $\alpha=1$ and $\beta=2$ in Figure~\ref{fig:diwheel1} completes the proof that $d(W_4)=3$.

When $\Delta\geq5$, by Theorem~\ref{thm:dilowerbnd} again, $d(W_\Delta)\geq\left\lceil\frac{\chi'(W_\Delta)}{2}\right\rceil\geq\left\lceil\frac{\Delta}{2}\right\rceil$. It remains to find a vertex labeling $f$ such that $|f^-|=\left\lceil\frac{\Delta}{2}\right\rceil$. Let $f$ be the vertex labeling on $W_\Delta$ such that $f(v_0)=0$, and for all $1\leq i\leq\Delta$,
$$f(v_i)=\begin{cases}
i&\text{if }1\leq i\leq\left\lceil\frac{\Delta}{2}\right\rceil-2,\\
\left\lceil\frac{\Delta}{2}\right\rceil&\text{if }i=\left\lceil\frac{\Delta}{2}\right\rceil-1,\\
\left\lceil\frac{\Delta}{2}\right\rceil-1&\text{if }i=\left\lceil\frac{\Delta}{2}\right\rceil,\\
\left\lceil\frac{\Delta}{2}\right\rceil-i& \text{if }\left\lceil\frac{\Delta}{2}\right\rceil+1\leq i\leq2\left\lceil\frac{\Delta}{2}\right\rceil-2,\\
-\left\lfloor\frac{\Delta}{2}\right\rfloor&\text{if }i=2\left\lceil\frac{\Delta}{2}\right\rceil-1,\\
1-\left\lfloor\frac{\Delta}{2}\right\rfloor&\text{if }i=2\left\lceil\frac{\Delta}{2}\right\rceil.\\
\end{cases}$$
Figures~\ref{fig:diwheel6} and \ref{fig:diwheel7} illustrate the vertex labelings $f$ on $W_6$ and $W_7$, respectively.
\begin{figure}[H]
\centering
\begin{minipage}{0.45\linewidth}
\centering
\begin{tikzpicture}
\coordinate(v0)at(0,0);\coordinate(v1)at({2*cos(0*360/6)},{2*sin(0*360/6)});\coordinate(v2)at({2*cos(1*360/6)},{2*sin(1*360/6)});\coordinate(v3)at({2*cos(2*360/6)},{2*sin(2*360/6)});\coordinate(v4)at({2*cos(3*360/6)},{2*sin(3*360/6)});\coordinate(v5)at({2*cos(4*360/6)},{2*sin(4*360/6)});\coordinate(v6)at({2*cos(5*360/6)},{2*sin(5*360/6)});\foreach\i in{v0,v1,v2,v3,v4,v5,v6}{\filldraw(\i)circle(0.1);}\draw(v1)--(v2)--(v3)--(v4)--(v5)--(v6)--(v1)--(v4);\draw(v2)--(v5);\draw(v3)--(v6);\node[below right]at(0.1,0.1){$0$};\node[above right]at(v1){$1$};\node[above right]at(v2){$3$};\node[above left]at(v3){$2$};\node[above left]at(v4){$-1$};\node[below left]at(v5){$-3$};\node[below right]at(v6){$-2$};
\end{tikzpicture}
\caption{Vertex labeling $f$ on $W_6$}
\label{fig:diwheel6}
\end{minipage}\qquad
\begin{minipage}{0.45\linewidth}
\centering
\begin{tikzpicture}
\coordinate(v0)at(0,0);\coordinate(v1)at({2*cos(0*360/7)},{2*sin(0*360/7)});\coordinate(v2)at({2*cos(1*360/7)},{2*sin(1*360/7)});\coordinate(v3)at({2*cos(2*360/7)},{2*sin(2*360/7)});\coordinate(v4)at({2*cos(3*360/7)},{2*sin(3*360/7)});\coordinate(v5)at({2*cos(4*360/7)},{2*sin(4*360/7)});\coordinate(v6)at({2*cos(5*360/7)},{2*sin(5*360/7)});\coordinate(v7)at({2*cos(6*360/7)},{2*sin(6*360/7)});\foreach\i in{v0,v1,v2,v3,v4,v5,v6,v7}{\filldraw(\i)circle(0.1);}\draw(v1)--(v2)--(v3)--(v4)--(v5)--(v6)--(v7)--(v1)--(v0)--(v2);\draw(v3)--(v0)--(v4);\draw(v5)--(v0)--(v6);\draw(v0)--(v7);\node[below right]at(0.1,0.1){$0$};\node[above right]at(v1){$1$};\node[above right]at(v2){$2$};\node[above left]at(v3){$4$};\node[above left]at(v4){$3$};\node[below left]at(v5){$-1$};\node[below right]at(v6){$-2$};\node[below right]at(v7){$-3$};
\end{tikzpicture}
\caption{Vertex labeling $f$ on $W_7$}
\label{fig:diwheel7}
\end{minipage}
\end{figure}

It is not difficult to check that the image of $f-$ is $\left\{1,2,\dotsc,\left\lceil\frac{\Delta}{2}\right\rceil\right\}$, which completes the proof that $d(W_\Delta)=\left\lceil\frac{\Delta}{2}\right\rceil$ when $\Delta\geq5$.
\end{proof}

\begin{theorem}\label{thm:digrid}
Let $G$ be a rectangular grid.  Then $d(G)=2$.
\end{theorem}

\begin{proof}
Let $G=L_{n\times m}$. If $n=m=2$, then $G$ is isomorphic to the cycle $C_4$. By Theorem~\ref{thm:dicycle}, we have $d(G)=d(C_4)=2$. Otherwise, if $n\geq3$, then $\Delta(G)\geq3$. By Theorem~\ref{thm:dilowerbnd}, $d(G)\geq\left\lceil\frac{\chi'(G)}{2}\right\rceil\geq\left\lceil\frac{\Delta(G)}{2}\right\rceil\geq\left\lceil\frac{3}{2}\right\rceil=2$. Hence, it suffices to find a vertex labeling $f$ such that $|f^-|=2$.

For all $0\leq i\leq n-1$ and $0\leq j\leq m-1$, define $f(v_{i,j})=mi+j$. Viewing $mi+j$ as the base $m$ expansion of a positive integer, we see that $f$ is injective. Furthermore, $|f^-|=2$ since
$$|f(v_{i,j})-f(v_{i+1,j})|=|mi+j-(m(i+1)+j)|=m$$
and
$$|f(v_{i,j})-f(v_{i,j+1})|=|mi+j-(mi+j+1)|=1.$$
Figure~\ref{fig:digrid} illustrates the vertex labeling $f$ on $L_{6\times3}$.
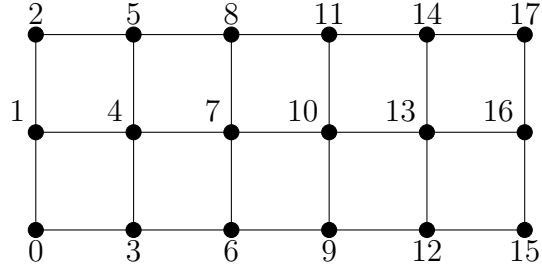
\begin{figure}[H]
\centering
\begin{tikzpicture}[scale=1.3]
\foreach\i in{0,1,2,3,4,5}{\foreach\j in{0,1,2}{\filldraw(\i,\j)circle(0.1/1.3);}}\draw(0,0)grid(5,2);\node[below]at(0,0){$0$};\node[below]at(1,0){$3$};\node[below]at(2,0){$6$};\node[below]at(3,0){$9$};\node[below]at(4,0){$12$};\node[below]at(5,0){$15$};\node[above left]at(0,1){$1$};\node[above left]at(1,1){$4$};\node[above left]at(2,1){$7$};\node[above left]at(3,1){$10$};\node[above left]at(4,1){$13$};\node[above left]at(5,1){$16$};\node[above]at(0,2){$2$};\node[above]at(1,2){$5$};\node[above]at(2,2){$8$};\node[above]at(3,2){$11$};\node[above]at(4,2){$14$};\node[above]at(5,2){$17$};
\end{tikzpicture}
\caption{Vertex labeling $f$ on $L_{6\times3}$}
\label{fig:digrid}
\end{figure}
\end{proof}

Since a prism graph $\Pi_n$ is a ladder graph $L_{n\times 2}$ with the two additional edges $v_{0,0}v_{n-1,0}$ and $v_{0,1}v_{n-1,1}$, the following corollary follows immediately from Theorem~\ref{thm:digrid} by using the same vertex labeling on ladder graphs.

\begin{corollary}\label{cor:diprism}
For every prism graph $\Pi_n$, $d(\Pi_n)\leq3$.
\end{corollary}

\subsection{The difference index of trees}

In Subsection~\ref{subsec:sumtrees}, we saw that every tree of sum index at most $k$ is isomorphic to a subgraph of the hyperdiamond $H_k$.  We will provide a similar treatment for the difference index in this subsection. This time, the role of the universal graph is played by the infinite $k$-dimensional rectangular grid.

\begin{definition}
Let $k$ be a positive integer. Let $\e_1,\e_2,\dotsc,\e_k$ be the standard basis vectors of $\Z^k$. The \emph{infinite $k$-dimensional rectangular grid} $Q_k$ is the Cayley graph of $\Z^k$ with generating set $S=\{\pm\e_1,\pm\e_2,\dotsc,\pm\e_k\}$. In other words, the vertex set of $Q_k$ is $\Z^k$, and for any $\x,\x'\in\Z^k$, there is a directed edge from $\x$ to $\x'$ if and only if $\x-\x'\in S$. Since $S$ is closed under inverses, the Cayley graph $Q_k$ is an undirected graph.
\end{definition}

From the definition, it is obvious that for all $\x,\x'\in \Z^k$, the distance between $\x$ and $\x'$ is given by $\lVert\x-\x'\rVert$. The following two theorems are analogous to Theorems~\ref{thm:sumsubgraph} and \ref{thm:treeembed}.%Furthermore, the map $v \mapsto v + v_0$ is an isomorphism of $SQ_k$.

\begin{theorem}
Let $G$ be a finite subgraph of $Q_k$. Then $d(G)\leq k$.
\end{theorem}

\begin{proof}
Let $r=\max\{\lVert\x\rVert:\x\in V(G)\}$. Define $f:V(G)\to\Z$ such that for all $\x=(x_1,x_2,\dotsc,x_k)\in V(G)$, $f(\x)=x_1+(2r+1)x_2+\dotsb+(2r+1)^{k-1}x_k$. Since $|x_i|\leq r$ for all $1\leq i\leq k$, by viewing $f(\x)$ as a base $2r+1$ expansion of an integer, it is easy to see that $f$ is injective. It remains to show that $|f^-|\leq k$.

Consider $\x\x'\in E(G)$, where $\x'=(x_1',x_2',\dotsc,x_k')$ and
$$\x=\pm\e_i+\x'=(x_1',x_2',\dotsc,\pm1+x_i',\dotsc,x_k')$$
for some $1\leq i\leq k$. Then
\begin{align*}
f^-(\x\x')&=|f(\x)-f(\x')|\\
&=\big|\big(x_1'+(2r+1)(x_2')+\dotsb+(2r+1)^{i-1}(\pm1+x_i')+\dotsb+(2r+1)^{k-1}(x_k')\big)\\
&\hspace{20pt}-\big(x_1'+(2r+1)x_2'+\dotsb+(2r+1)^{k-1}x_k'\big)\big|\\
&=(2r+1)^{i-1}.
\end{align*}
Therefore, $f^-(E(G))\subseteq\{1,2r+1,(2r+1)^2,\dotsc,(2r+1)^{k-1}\}$.
\end{proof}

\begin{theorem}\label{thm:ditreeembed}
If $G$ is a tree and $d(G)\leq k$, then $G$ is isomorphic to a subgraph of $Q_k$.
\end{theorem}

\begin{proof}
Since $d(G)\leq k$, let $f$ be a vertex labeling of $G$ such that $f^-(E(G))\subseteq\{\alpha_1,\alpha_2,\dotsc,\alpha_k\}$. Let $\mathcal{L}:\{\pm\alpha_1,\pm\alpha_2,\dotsc,\pm\alpha_k\}\to\{\pm\e_1,\pm\e_2,\dotsc,\pm\e_k\}$ be such that $\mathcal{L}(\alpha_i)=\e_i$ and $\mathcal{L}(-\alpha_i)=-\e_i$ for all $1\leq i\leq k$.

Fix a vertex $v_0\in V(G)$. By Corollary~\ref{cor:di0vertex}, we may assume that $f(v_0)=0$. For any vertex $w\in V(G)$ of distance $r$ away from $v_0$, there exists a unique path $v_0v_1v_2\dotsb v_r$ from $v_0$ to $w$, where $v_r=w$. We define a map $\Lambda:V(G)\to V(Q_k)=\Z^k$ such that 
$$\Lambda(w)=\mathcal{L}(f(v_r)-f(v_{r-1}))+\mathcal{L}(f(v_{r-1})-f(v_{r-2}))+\dotsb+\mathcal{L}(f(v_1)-f(v_0)).$$

To verify that $\Lambda$ is a graph homomorphism, consider two adjacent vertices $u$ and $w$ in $V(G)$. Then their distances away from $v_0$ differ by $1$. Without loss of generality, let the paths from $v_0$ to $u$ and from $v_0$ to $w$ be $v_0v_1v_2\dotsb v_rv_{r+1}$ and $v_0v_1v_2\dotsb v_r$, respectively, where $v_{r+1}=u$ and $v_r=w$. Hence, $\Lambda(u)=\mathcal{L}(f(v_{r+1})-f(v_r))+\Lambda(w)$, or $\Lambda(u)-\Lambda(w)=\mathcal{L}(f(v_{r+1})-f(v_r))$. This shows that $\Lambda(u)$ and $\Lambda(w)$ are adjacent in $Q_k$, since $\mathcal{L}(f(v_{r+1})-f(v_r))$ is in the generating set of $Q_k$.

It remains to verify that $\Lambda$ is injective. We define a linear map $\mathcal{T}:\Z^k\to\Z$ such that
$$\mathcal{T}(x_1,x_2,\dotsc,x_k)=\alpha_1x_1+\alpha_2x_2+\dotsb\alpha_kx_k.$$
Note that for every edge $uv\in E(G)$, if $f(u)-f(v)=\alpha_i$ or $-\alpha_i$, then $\mathcal{L}(f(u)-f(v))=\e_i$ or $-\e_i$, respectively. As a result, $\mathcal{T}\big(\mathcal{L}(f(u)-f(v))\big)=f(u)-f(v)$.

Suppose $\Lambda(u)=\Lambda(w)$ for some vertices $u,w\in V(G)$. Let the paths from $v_0$ to $u$ and from $v_0$ to $w$ be $u_0u_1u_2\dotsb u_r$ and $w_0w_1w_2\dotsb w_s$, respectively, where $u_0=w_0=v_0$, $u_r=u$, and $w_s=w$. Then 
\begin{align*}
\mathcal{T}(\Lambda(u))&=\mathcal{T}\big(\mathcal{L}(f(u_r)-f(u_{r-1}))+\mathcal{L}(f(u_{r-1})-f(u_{r-2}))+\dotsb+\mathcal{L}(f(u_1)-f(u_0))\big)\\
&=\mathcal{T}\big(\mathcal{L}(f(u_r)-f(u_{r-1}))\big)+\mathcal{T}\big(\mathcal{L}(f(u_{r-1})-f(u_{r-2}))\big)+\dotsb+\mathcal{T}\big(\mathcal{L}(f(u_1)-f(u_0))\big)\\
&=(f(u_r)-f(u_{r-1}))+(f(u_{r-1})-f(u_{r-2}))+\dotsb+(f(u_1)-f(u_0))\\
&=f(u_r).
\end{align*}
Similarly, $\mathcal{T}(\Lambda(w))=f(w_s)$. Since $\mathcal{T}(\Lambda(u))=\mathcal{T}(\Lambda(w))$, we have $f(u_r)=f(w_s)$. By the injectivity of $f$, we conclude that $u_r=w_s$, i.e., $u=w$. Therefore, $\Lambda$ is injective, thus $G$ is isomorphic to a subgraph of $Q_k$, as desired.
\end{proof}

\begin{theorem}
Let $r$ be a positive integer. Then the number of vertices in $Q_k$ that are of distance $r$ away from a fixed vertex $\textup{\textbf{v}}\in V(Q_k)$ is $$\sum_{j = 1}^k\binom{k}{j}\binom{r-1}{j-1}2^j.$$
\end{theorem}

\begin{proof}
Since $Q_k$ is vertex transitive, we may assume that $\textup{\textbf{v}}=(0,0,\dotsc,0)\in V(Q_k)$. Let $\mathcal{V}$ be the set of vertices in $Q_k$ that are of distance $r$ away from $\textup{\textbf{v}}$. Then
$$\mathcal{V}=\{(x_1,x_2,\dotsc,x_k)\in\Z^k:|x_1|+|x_2|+\dotsb+|x_k|=r\}.$$
To count the number of elements in $\mathcal{V}$, we first partition $\mathcal{V}$ into $\mathcal{V}_1,\mathcal{V}_2,\dotsc,\mathcal{V}_k$, where
$$\mathcal{V}_j=\{\x=(x_1,x_2,\dotsc,x_k)\in\mathcal{V}:\text{the number of nonzero entries in $\x$ is }j\}.$$
For each fixed $j\in\{1,2,\dotsc,k\}$, there are $\binom{k}{j}$ ways to partition $\{1,2,\dotsc,k\}$ into two subsets $S_1$ and $S_2$ such that $|S_1|=j$ and $|S_2|=k-j$. For each such partition $S_1$ and $S_2$, let
$$\mathcal{V}_{S_1}=\{\x\in\mathcal{V}_j:x_i\neq0\text{ if and only if }i\in S_1\}.$$
It is easy to see that the number of elements in $\mathcal{V}_{S_1}$ is the same as the number of nonzero integer solutions $(x_1',x_2',\dotsc,x_j')$ to the equation
$$\sum_{i=1}^j|x_i'|=r.$$
The number of nonzero solutions to this equation is $\binom{r-1}{j-1}2^j$. Therefore, the cardinality of $\mathcal{V}_j$ is $\binom{k}{j}\binom{r-1}{j-1}2^j$, thus proving the theorem by summing over $j\in\{1,2,\dotsc,k\}$.
\end{proof}

Similarly to Corollary~\ref{cor:tree}, we provide without proof a lower bound for the difference index of a tree.

\begin{corollary}
Let $G$ be a tree. Then $d(G)\geq k$, where $k$ is the minimum positive integer such that for every vertex $v\in V(G)$ and positive integer $r$, the number of vertices in $G$ that are of distance $r$ away from $v$ is at most
$$\sum_{j=1}^k \binom{k}{j}\binom{r-1}{j-1}2^j.$$
\end{corollary}

\section{Concluding remarks}

We showed in Section~\ref{sec:sumindex} that the exclusive sum number is an upper bound for the sum index of a graph.  In Theorem~\ref{thm:esnbound} we further showed that there exists a graph $G$ such that $s(G)<\epsilon(G)$.  Preliminary investigations on such graphs lead us to the following conjecture.

\begin{conjecture}
For all positive integers $N$, there exists a graph $G$ such that $\epsilon(G)-s(G)>N$.
\end{conjecture}

Table~\ref{tbl:sumdiff} compares the sum index to the difference index for the various families of graphs studied in this paper.  
\begin{table}[H]
\centering
    \begin{tabular}{|l|c|c|}
    \hline  & Sum Index & Difference Index \\
    \hline Complete graphs ($K_n$, where $n\geq2$) & $2n-3$ & $n-1$ \\
    \hline Complete bipartite graphs ($K_{n,m}$) & $m+n-1$ & $\left\lceil\frac{m+n-1}{2}\right\rceil$ \\
    \hline Caterpillars & $\Delta$ & $\left\lceil\frac{\Delta}{2}\right\rceil$ \\
    \hline Cycles ($C_n$) & $3$ & $2$ \\
    \hline Spiders ($S_{\ell_1,\ell_2,\dotsc,\ell_\Delta}$) & $\Delta$ & $\left\lceil\frac{\Delta}{2}\right\rceil$ \\ 
    \hline Wheels ($W_3,W_4$) & $5$ & $3$ \\
    \hline Wheels ($W_\Delta$, where $\Delta\geq5$) & $\Delta$ & $\left\lceil\frac{\Delta}{2}\right\rceil$ \\
    \hline Rectangular grids ($L_{n\times m}$, where $3\leq n\leq m$) & $4$ & $2$ \\
    \hline Ladders ($L_{2\times m}$, where $m\geq2$) & $3$ & $2$ \\
    \hline
    \end{tabular}
    \caption{Comparing sum index and difference index}
    \label{tbl:sumdiff}
\end{table}
    
\noindent We note that in all cases of Table~\ref{tbl:sumdiff}, we have \begin{equation}\label{eqn:d(g)s(g)}
d(G)=\left\lceil\frac{s(G)}{2}\right\rceil,
\end{equation} and it is tempting to conjecture that \eqref{eqn:d(g)s(g)} holds for all nonempty simple graph $G$. In fact, it can be verified that \eqref{eqn:d(g)s(g)} is true for all nonempty simple graphs with up to $4$ vertices.  However, the following two examples show that equation \eqref{eqn:d(g)s(g)} does not always hold.

\begin{example}
Let $G$ be the graph shown in Figure~\ref{fig:pentagon}.
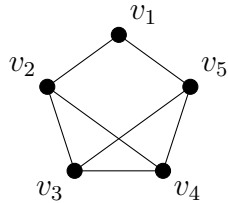
\begin{figure}[H]
\centering
\begin{tikzpicture}
\foreach\i in{0,1,2,3,4}{\filldraw({cos(72*\i+18)},{sin(72*\i+18)})circle(0.1);\draw({cos(72*\i+18)},{sin(72*\i+18)})--({cos(72*(\i+1)+18)},{sin(72*(\i+1)+18)});}\draw({cos(18)},{sin(18)})--({cos(234)},{sin(234)});\draw({cos(162)},{sin(162)})--({cos(306)},{sin(306)});\node[above right]at({cos(18)},{sin(18)}){$v_5$};\node[above right]at({cos(90)},{sin(90)}){$v_1$};\node[above left]at({cos(162)},{sin(162)}){$v_2$};\node[below left]at({cos(234)},{sin(234)}){$v_3$};\node[below right]at({cos(306)},{sin(306)}){$v_4$};
\end{tikzpicture}
\caption{Graph with $d(G)>\left\lceil\frac{s(G)}{2}\right\rceil$}\label{fig:pentagon}
\end{figure}
If $(f(v_1),f(v_2),f(v_3),f(v_4),f(v_5))=(1,5,2,3,4)$ and $g(v_1),g(v_2),g(v_3),g(v_4),g(v_5))=(1,2,4,5,3)$, then $|f^+|=4$ and $|g^-|=3$, thus $s(G)\leq4$ and $d(G)\leq3$. On the other hand, since the induced subgraph $H$ on $\{v_2,v_3,v_4,v_5\}$ is isomorphic to the complete graph $K_4$ with an edge removed, and $s(K_4)=2\times4-3=5$, we deduce that $s(G)\geq s(H)\geq5-1=4$. Hence, $s(G)=4$. In the rest of this example, we are going to show that $d(G)>2$.

Assume for the sake of contradiction that $h$ is a difference index labeling of $G$, where the image set of $h^-$ is the set $\{\alpha,\beta\}$ and $0<\alpha<\beta$. Then we have either $h^-(v_2v_3)=h^-(v_3v_4)$ or $\{h^-(v_2v_3),h^-(v_3v_4)\}=\{\alpha,\beta\}$. By Corollary~\ref{cor:di0vertex}, we may assume that $h(v_3)=0$.

Case $1$: $h^-(v_2v_3)=h^-(v_3v_4)$. Since $h(v_3)=0$, we have either $\{h(v_2),h(v_4)\}=\{\alpha,-\alpha\}$ or $\{h(v_2),h(v_4)\}=\{\beta,-\beta\}$. If $\{h(v_2),h(v_4)\}=\{\beta,-\beta\}$, then $h^-(v_2v_4)=2\beta>\beta>\alpha$, contradicting that the image of $h^-$ is $\{\alpha,\beta\}$. Hence, $\{h(v_2),h(v_4)\}=\{\alpha,-\alpha\}$, and without loss of generality, we may assume that $h(v_2)=\alpha=-h(v_4)$. Moreover, $h^-(v_2v_4)=2\alpha$, thus $\beta=2\alpha$. Since $h(v_3)=0$, $|h(v_5)|=h^-(v_3v_5)\in\{\alpha,\beta\}$. However, $h(v_5)\notin\{h(v_2),h(v_4)\}=\{\alpha,-\alpha\}$, so $h(v_5)\in\{\beta,-\beta\}=\{2\alpha,-2\alpha\}$, and it is easy to see that $h(v_5)=-2\alpha$. If $h(v_1)<-\alpha$, then $h^-(v_1v_2)=h(v_2)-h(v_1)>2\alpha$; if $-\alpha<h(v_1)<0$, then $\alpha<h^-(v_1v_2)=h(v_2)-h(v_1)<2\alpha$; if $h(v_1)>0$, then $h^-(v_1v_5)=h(v_1)-h(v_5)>2\alpha$. Therefore, $h(v_1)\in\{-\alpha,0\}$, contradicting that $h$ is injective.

Case $2$: $\{h^-(v_2v_3),h^-(v_3v_4)\}=\{\alpha,\beta\}$. Since $h(v_3)=0$, we have either $h^-(v_2v_4)=\beta-\alpha$ or $h^-(v_2v_4)=\alpha+\beta$. Note that $\alpha+\beta>\beta>\alpha$ and $h^-(v_2v_4)\in\{\alpha,\beta\}$, so $h^-(v_2v_4)\neq\alpha+\beta$. Hence, $h^-(v_2v_4)=\beta-\alpha$, which implies that $\beta=2\alpha$. Without loss of generality, we may assume that $\{h(v_2),h(v_4)\}=\{\alpha,\beta\}$. Since $h(v_3)=0$, $|h(v_5)|=h^-(v_3v_5)\in\{\alpha,\beta\}$. However, $h(v_5)\notin\{h(v_2),h(v_4)\}=\{\alpha,\beta\}$, so $h(v_5)\in\{-\alpha,-\beta\}$. If $h(v_4)=\beta$, then $h^-(v_4v_5)=h(v_4)-h(v_5)>\beta>\alpha$, contradicting that the image of $h^-$ is $\{\alpha,\beta\}$. Therefore, $h(v_4)=\alpha$, $h(v_2)=\beta=2\alpha$, and $h(v_5)=-\alpha$. If $h(v_1)<0$, then $h^-(v_1v_2)=h(v_2)-h(v_1)>2\alpha$; if $0<h(v_1)<\alpha$, then $\alpha<h^-(v_1v_2)=h(v_2)-h(v_1)<2\alpha$; if $h(v_1)>\alpha$, then $h^-(v_1v_5)=h(v_1)-h(v_5)>2\alpha$. Therefore, $h(v_1)\in\{0,\alpha\}$, contradicting that $h$ is injective.
\end{example}

\begin{example}
Let $G$ be the tree show in Figure~\ref{fig:tree}.
\begin{figure}[H]
\centering
\begin{tikzpicture}
\filldraw(0,0)circle(0.1);\foreach\i in{0,1,2,3}{\filldraw({cos(\i*90+45)},{sin(\i*90+45)})circle(0.1);\filldraw({2*cos(\i*90+45)},{2*sin(\i*90+45)})circle(0.1);\filldraw({cos(\i*90+45)+cos(\i*90)},{sin(\i*90+45)+sin(\i*90)})circle(0.1);\filldraw({cos(\i*90+45)+cos((\i+1)*90)},{sin(\i*90+45)+sin((\i+1)*90)})circle(0.1);
\draw(0,0)--({2*cos(\i*90+45)},{2*sin(\i*90+45)});\draw({cos(\i*90+45)+cos(\i*90)},{sin(\i*90+45)+sin(\i*90)})--({cos(\i*90+45)},{sin(\i*90+45)})--({cos(\i*90+45)+cos((\i+1)*90)},{sin(\i*90+45)+sin((\i+1)*90)});}
\end{tikzpicture}
\caption{}\label{fig:tree}
\end{figure}
Since the hyperdiamonds $H_3$ and $H_4$ are regular of degree $3$ and $4$, respectively, and both hyperdiamonds have girth $6$, it is easy to see that $G$ is isomorphic to a subgraph of $H_4$ but not $H_3$. By Theorems~\ref{thm:sumsubgraph} and \ref{thm:treeembed}, $s(G)=4$. On the other hand, it is easy to see that $G$ is not isomorphic to a subgraph of the rectangular grid $Q_2$, so $d(G)>2$ by Theorem~\ref{thm:ditreeembed}.
\end{example}

These observations together with Theorem~\ref{thm:dibipartite} lead us to the following conjecture.

\begin{conjecture}
For any nonempty simple graph $G$, $\left\lceil\frac{s(G)}{2}\right\rceil\leq d(G)\leq s(G)$.
\end{conjecture}

We provided only two conjectures in this section, but since the definitions of sum index and difference index are fairly new, there are many other exciting directions of study that the reader may develop on these topics.

\section{Acknowledgments}
These results are based on work supported by the National Science Foundation under grants numbered DMS-1852378 and DMS-1560019.

\end{document}